\documentclass[12pt]{article}

\usepackage{hyperref}

\usepackage{amssymb}
\usepackage{amsmath}
\usepackage{amsthm}

\newtheorem{thm}{Theorem}[section]
\newtheorem{lem}[thm]{Lemma}

\newtheorem{cor}[thm]{Corollary}

\newtheorem*{thm*}{Theorem}

\theoremstyle{definition}
\newtheorem{defn}[thm]{Definition}
\newtheorem{example}[thm]{Example}
\newtheorem{question}[thm]{Question}

\theoremstyle{remark}

\newcommand{\system}[1]{\mbox{\fontfamily{cmss}\fontshape{n}\fontseries{m}%
    \selectfont#1}}
\newcommand{\RCA}{\system{RCA}\ensuremath{_0}}

\newcommand{\RT}{\system{RT}\ensuremath{^2_2}}
\newcommand{\SRT}{\system{SRT}\ensuremath{^2_2}}

\newcommand{\ADS}{\system{ADS}}
\newcommand{\CAC}{\system{CAC}}
\newcommand{\SCAC}{\system{SCAC}}

\newcommand{\COH}{\system{COH}}
\newcommand{\EM}{\system{EM}}
\newcommand{\OPT}{\system{OPT}}
\newcommand{\SADS}{\system{SADS}}
\newcommand{\SEM}{\system{SEM}}

\title{Separating principles below Ramsey's Theorem for Pairs}

\author{Manuel Lerman, Reed Solomon, Henry Towsner\thanks{Henry Towsner was partially supported by NSF grant DMS-1157580.}}

\begin{document}

\maketitle

\section{Introduction}
\label{sec:intro}

In recent years, there has been a substantial amount of work in reverse mathematics concerning natural mathematical principles that are 
provable from $\RT$, Ramsey's Theorem for Pairs.  These principles tend to fall outside of the ``big five" systems of reverse mathematics and 
a complicated picture of subsystems below $\RT$ has emerged.  In this paper, we answer two open questions concerning these subsystems, specifically that 
$\ADS$ is not equivalent to $\CAC$ and that $\EM$ is not equivalent to $\RT$.  

We begin with a review of the definitions and known results for the 
relevant systems below $\RT$, but will assume a general familiarity with reverse mathematics.  We refer the reader to Simpson \cite{Sim99} 
for background on reverse mathematics and to Hirschfeldt and Shore \cite{Hir07} for background on the general picture of subsystems below $\RT$.  
Unless otherwise specified, we always work in the base theory $\RCA$.  

We will have orderings on a variety of structures, but we typically reserve the symbols $<$ and $\leq$ for three contexts:  the usual order on $\mathbb{N}$, 
extensions of forcing conditions and comparing sets.  If $F$ is a finite set and $G$ is a (finite or infinite) set, we write $F < G$ to denote 
$\text{max}(F) < \text{min}(G)$.  Without loss of generality, we assume that the infinite algebraic structures defined below have domain $\mathbb{N}$.  

\begin{defn}
A \emph{2-coloring} of $[\mathbb{N}]^2$ (or simply a \emph{coloring}), where $[\mathbb{N}]^2$ denotes the set of all two element subsets of 
$\mathbb{N}$, 
is a function $c:[\mathbb{N}]^2 \rightarrow \{ 0,1 \}$.   A set $H \subseteq \mathbb{N}$ is \emph{homogeneous} for $c$ if $c$ is constant on $[H]^2$.    
\end{defn}

(\RT) Ramsey's Theorem for Pairs: Every 2-coloring of $[\mathbb{N}]^2$ has an infinite homogeneous set.   

\medskip
We refer to an infinite homogeneous set for a coloring $c$ as a \textit{solution} to $c$.  We typically write $c(x,y)$, as opposed to 
$c(\{x,y\})$, with implicit assumption that $x < y$.  

\begin{defn}
Let $c$ be a 2-coloring of $[\mathbb{N}]^2$.  Define $A^*(c)$ (respectively $B^*(c)$) to be the set of numbers which are colored $0$ (respectively $1$) 
with all but finitely many other numbers.   
\begin{eqnarray*}
A^*(c) & = & \{ n \mid \exists x \, \forall y > x \, (c(n,y) = 0) \} \\
B^*(c) & = & \{ n \mid \exists x \, \forall y > x \, (c(n,y) = 1) \}
\end{eqnarray*}
The coloring $c$ is \emph{stable} if $A^*(c) \cup B^*(c) = \mathbb{N}$.  
\end{defn}

(\SRT) Stable Ramsey's Theorem for Pairs:  Every stable 2-coloring of $[\mathbb{N}]^2$ has an infinite homogeneous set.

\medskip
Chong, Slaman and Yang \cite{Cho:TA} have recently shown that $\SRT$ is strictly weaker than $\RT$.

\begin{defn}
Let $M = (\mathbb{N}, \preceq_M)$ be a poset.  For $x,y \in M$, we say that $x$ and $y$ are \emph{comparable} if either 
$x \preceq_M y$ or $y \preceq_M x$, and we say $x$ and $y$ are \emph{incomparable} (and write $x \, |_M \, y$) if $x \not \preceq_M y$ and 
$y \not \preceq_M x$. $S \subseteq \mathbb{N}$ is a \emph{chain} in $M$ if for all $x,y \in S$, $x$ and $y$ are comparable.  $S$ is 
an \emph{antichain} in $M$ if for all $x \neq y \in S$, $x$ and $y$ are incomparable. 
\end{defn}

(\CAC) Chain-AntiChain: Every infinite poset $M$ contains either an infinite chain or an infinite antichain. 

\medskip
A \emph{solution} to an infinite poset $M$ is an infinite set $S$ such that $S$ is either a chain or an antichain.  
It is straightforward to show that $\RT \vdash \CAC$ by transforming instances of $\CAC$ into instances 
of $\RT$.  Given a partial order $M = (\mathbb{N}, \preceq_M)$, define the coloring 
$c_M$ by setting $c_M(x,y) = 0$ if $x$ and $y$ are comparable and setting $c_M(x,y) = 1$ otherwise.  If $H$ is an infinite homogeneous set for 
$c_M$ with color $0$, then $H$ is an infinite chain in $M$.  If $H$ is an infinite homogeneous set with color $1$, then $H$ is an infinite antichain in $M$.  

Hirschfeldt and Shore \cite{Hir07} showed that one cannot give a similar transformation of instances of $\RT$ into instances of $\CAC$ 
by showing that $\CAC \not \vdash \SRT$.  

\begin{defn}
Let $M = (\mathbb{N}, \preceq_M)$ be an infinite partial order.  Define 
\begin{eqnarray*}
A^*(M) & = & \{ n \mid \exists x \, \forall y > x \, (n \preceq_M y) \} \\
B^*(M) & = & \{ n \mid \exists x \, \forall y > x \, (n \, |_M \, y ) \} \\
C^*(M) & = & \{ n \mid \exists x \, \forall y > x \, (y \preceq_M n) \} 
\end{eqnarray*}
$M$ is \emph{stable} if either $A^*(M) \cup B^*(M) = \mathbb{N}$ or $C^*(M) \cup B^*(M) = \mathbb{N}$.  
\end{defn}

(\SCAC) Stable Chain-Antichain: Every infinite stable poset $M$ contains either an infinite chain or an infinite anti chain.

\medskip
When we work with $\SCAC$ later, we will construct an infinite poset $M$ such that $A^*(M) \cup B^*(M) = \mathbb{N}$.  Thus, our notations for 
$A^*(M)$ and $B^*(M)$ are chosen to parallel the corresponding notations for $\SRT$.  Although $\SRT \vdash \SCAC$ by the transformation given above, 
Hirschfeldt and Shore \cite{Hir07} showed that $\SCAC \not \vdash \CAC$.

\begin{defn}
Let $L = (\mathbb{N}, <_L)$ be an infinite linear order.  A function $f:\mathbb{N} \rightarrow L$ is an \emph{infinite ascending sequence} in $L$ 
if for all $n < m$, $f(n) <_L f(m)$ and is an \emph{infinite descending sequence} in $L$ if for all $n < m$, $f(n) >_L f(m)$.
\end{defn}

(\ADS) Ascending or Descending Sequence: Every infinite linear order $L$ has an infinite ascending sequence or an 
infinite descending sequence.  

\begin{defn}
An infinite linear order $L$ is \textit{stable} if $L$ has order type $\omega+\omega^*$.  That is, for every $x$, there is a $y$ such that either 
$\forall z > y \, (x <_L z)$ or $\forall z > y \, (z <_L x)$.
\end{defn}

(\SADS) Stable Ascending or Descending Sequence:  Every infinite stable linear order has an infinite ascending sequence or an infinite 
descending sequence.

\medskip
A \emph{solution} to an infinite linear order $L$ is a function which is either an infinite ascending sequence or an infinite descending sequence.  

As above, one can show $\CAC \vdash \ADS$ by transforming instances of $\ADS$ into instances of $\CAC$.  Given an infinite linear order 
$(\mathbb{N}, <_L)$, define an infinite partial order $M = (\mathbb{N}, \preceq_M)$ by $x \preceq_M y$ if and only if $x \leq_L y$ and $x \leq y$.  Let 
$S = \{ s_0 < s_1 < \cdots \}$ be a solution to $M$ and define $f(n) = s_n$.  If $S$ is a chain in $M$, then $f$ is an ascending chain in $L$.  If $S$ is an antichain in 
$M$, then $f$ is a descending chain in $L$. 

Hirschfeldt and Shore \cite{Hir07} proved that $\SADS \not \vdash \ADS$, but left open the question of whether $\ADS$ implies $\CAC$ or 
$\SADS$ implies $\SCAC$.   Our first result answers both of these questions in the negative by separating $\ADS$ from $\SCAC$ in an 
$\omega$-model.  

\begin{thm}
\label{thm:ADS}
There is a Turing ideal $\mathcal{I} \subseteq \mathcal{P}(\omega)$ such that the $\omega$-model $(\omega, \mathcal{I})$ satisfies 
$\ADS$ but not $\SCAC$.  Therefore, $\ADS$ does not imply $\SCAC$.
\end{thm}

This theorem will be proved in Section \ref{sec:ADS}.  Our second result concerns infinite tournaments and the Erd\"{o}s-Moser Theorem.  

\begin{defn}
A \emph{tournament} $T$ on a domain $D \subseteq \mathbb{N}$ is an irreflexive binary relation on $D$ such that for all $x \neq y \in D$, exactly 
one of $T(x,y)$ or $T(y,x)$ holds.  $T$ is \emph{transitive} if for all $x,y,z \in D$, if $T(x,y)$ and $T(y,z)$ hold, then 
$T(x,z)$ holds.   
\end{defn}

In keeping with our terminology above, an \textit{infinite tournament} refers to a tournament $T$ with domain $\mathbb{N}$.  An 
\textit{infinite transitive subtournament of} $T$ (or a \textit{solution} to $T$) is an infinite set $S \subseteq \mathbb{N}$ such that $T$ restricted to domain $S$ 
is transitive.  The Erd\"{o}s-Moser Principle states that such solutions always exist.

\medskip
$(\EM)$ Erd\"{o}s-Moser Principle: Every infinite tournament contains an infinite transitive subtournament.  
\medskip

$\EM$ follows from $\RT$ by transforming instances of $\EM$ into instances of $\RT$.  Let $T$ be an 
infinite tournament and define the coloring $c_T$ for $x < y$ by $c_T(x,y) = 0$ if $T(x,y)$ holds and $c_T(x,y) = 1$ if $T(y,x)$ holds.  Suppose $H$ 
is an infinite homogeneous set for the color $0$.  Then, $H$ is transitive in $T$ because for all $x \neq y \in H$, $T(x,y)$ holds if and 
only if $x < y$.  Similarly, if $H$ is homogeneous for the color $1$, then $H$ is transitive in $T$ because for all $x \neq y \in H$, $T(x,y)$ holds 
if and only if $x > y$.  

Since computable instances of $\RT$ have $\Pi^0_2$ solutions and have $\text{low}_2$ solutions, it follows from this 
translation that computable instances of $\EM$ also have $\Pi^0_2$ solutions and have $\text{low}_2$ solutions.  
In Section \ref{sec:Erdos}, we present a proof due to Kach, Lerman, Solomon and Weber that these bounds are best possible.  

\begin{thm}[Kach, Lerman, Solomon and Weber]
\label{thm:KLSW}
There is a computable instance of $\EM$ with no $\Delta^0_2$ solution, and hence no $\Sigma^0_2$ solution or low solution.  
\end{thm}

Similar techniques were used by Dzhafarov, Kach, Lerman and Solomon to diagonalize against the existence of hyperimmune-free solutions.

\begin{thm}[Dzhafarov, Kach, Lerman and Solomon]
\label{thm:DKLS}
There is a computable instance of $\EM$ with no hyperimmune-free solution.
\end{thm}

Formalizing Theorem \ref{thm:DKLS} in reverse mathematics, which can be done in $\RCA + B\Sigma^0_2$, gives a lower bound on the strength of 
$\EM$.  Hirschfeldt, Shore and Slaman \cite{Hir09} proved that the following version of the Omitting Types Theorem, denoted $\OPT$, is equivalent to 
the statement that for every $X$, there is a function not dominated by any $X$-recursive function (i.e.~there is a degree which is hyperimmune relative to $X$).

\medskip

(\OPT) Omitting Partial Types: Let $T$ be a complete theory and $S$ be a set of partial types of $T$.  There is a 
model of $T$ that omits all the nonprincipal types in $S$. 

\medskip

Hence, $\EM$ implies $\OPT$ over $\RCA+B\Sigma^0_2$.  It remains an open question whether $\EM$ implies $B\Sigma^0_2$.

Bovykin and Weiermann \cite{BovTA} showed that one can transform an instance $c$ of $\RT$ into an instance $T_c$ of $\EM$, but that extracting the solution 
to $c$ from the solution to $T_c$ requires an application of $\ADS$.  To see why $\ADS$ might be useful, notice that if $S$ is a transitive subtournament of 
an infinite tournament $T$, then $T$ defines a linear order on $S$.  

\begin{thm}[Bovykin and Weiermann \cite{BovTA}]
\label{thm:BW}
$\EM + \ADS$ implies $\RT$.  
\end{thm}

\begin{proof}
Fix a coloring $c:[\mathbb{N}]^2 \rightarrow \{ 0,1 \}$.  Define an infinite tournament $T_c$ as follows.  $T_c(x,y)$ holds if either 
$x < y$ and $c(x,y) = 1$ or $y < x$ and $c(y,x) = 0$.  Let $S$ be an infinite transitive subtournament of $T_c$ and let $\leq_S$ be the linear 
order on $S$ induced by $T_c$.  By $\ADS$, let $f$ be an infinite ascending sequence or an infinite descending sequence in $(S, \leq_S)$.  
By thinning out $f$, we can assume that $f(0) < f(1) < f(2) < \cdots$ and hence the range of $f$ exists in $\RCA$.  

Suppose that $f$ is an ascending sequence in $\leq_S$.  Fix $n < m$.  Since $f(n) <_S f(m)$, the relation $T_c(f(n),f(m))$ holds.  Because 
$f(n) < f(m)$ and $T_c(f(n),f(m))$ holds, it follows that $c(f(n),f(m)) = 1$.  Therefore, the range of $f$ is homogeneous for $c$ with color $1$.

Suppose that $f$ is a descending sequence in $\leq_S$.  Fix $n < m$.  Since $f(m) <_S f(n)$, the relation $T_c(f(m),f(n))$ holds.  Because 
$f(n) < f(m)$, it follows that $c(f(n),f(m)) = 0$.  Therefore, the range of $f$ is homogeneous for $c$ with color 0.
\end{proof}

\begin{cor}
\label{cor:CAC_EM}
$\CAC$ does not prove $\EM$ (and hence $\ADS$ does not prove $\EM$ either).
\end{cor}

\begin{proof}
Suppose for a contradiction that $\CAC$ implies $\EM$.  Since $\CAC$ also proves $\ADS$, it follows from Theorem \ref{thm:BW} that 
$\CAC$ proves $\RT$.  However, by Hirschfeldt and Shore \cite{Hir07}, $\CAC$ does not prove $\RT$.  
\end{proof}

\begin{cor}
\label{cor:EM_ADS}
$\EM$ implies $\RT$ if and only if $\EM$ implies $\ADS$.  
\end{cor}

\begin{proof}
This follows immediately from Theorem \ref{thm:BW} and the fact that $\RT$ implies $\ADS$.  
\end{proof}

An infinite tournament $T$ is \textit{stable} if for all $x$, there is a $y$ such that either $T(x,z)$ holds for all $z > y$ or $T(z,x)$ holds for all $z > y$.  

\medskip
(\SEM) Stable Erd\"{o}s-Moser Principle: Every infinite stable tournament contains an infinite transitive subtournament.

\begin{cor}
$\SEM + \SADS$ implies $\SRT$.
\end{cor}

\begin{proof}
Let $c$ be a stable coloring and define $T_c$ as in Theorem \ref{thm:BW}.  We show that $T_c$ is a stable tournament.  

Fix $x$.  Let $y > x$ and $i \in \{ 0,1 \}$ be such that $c(x,z) = i$ for all $z > y$.  
Suppose that $i = 0$.  For every $z > y$, we have $x < z$ and $c(x,z) = 0$, and hence $T_c(z,x)$ holds.  On the other hand, suppose $i=1$.  For all $z > y$,   
we have $x < z$ and $c(x,z) = 1$, we have $T_c(x,z)$ holds.  Therefore, $T_c$ is stable.

By $\SEM$, there is an infinite transitive subtournament $S$ of $T_c$.  The corollary follows once we show that the linear order induced by $T_c$ on $S$ 
is stable.  Fix $x \in S$.  Since $T_c$ is stable, there is a $y > x$ such that either $T_c(x,z)$ holds for all $z > y$ (and hence $x <_S z$ for all $z > y$ with 
$z \in S$) or $T_c(z,x)$ holds for all $z > y$ (and hence $z <_S x$ for all $z > y$ with $z \in S$).  Therefore, $(S,\leq_S)$ is a stable linear order and 
$\SADS$ suffices to extract an infinite ascending or descending chain in $S$.
\end{proof}  

Our second result, to be proved in Section \ref{sec:EM}, is that $\EM$ does not imply $\SRT$, and hence the inclusion of $\ADS$ in Theorem \ref{thm:BW} 
cannot be removed.

\begin{thm}
\label{thm:EM}
There is a Turing ideal $\mathcal{I} \subseteq \mathcal{P}(\omega)$ such that the $\omega$-model $(\omega,\mathcal{I})$ satisfies $\EM$ but not 
$\SRT$.  Therefore, $\EM$ does not imply $\SRT$.
\end{thm}

\begin{cor}
\label{cor:EM}
$\EM$ does not imply $\SADS$ (and hence neither $\EM$ nor $\SEM$ implies either $\ADS$ or $\SADS$).
\end{cor}

\begin{proof}
Suppose for a contradiction that $\EM$ implies $\SADS$.  Since $\EM$ implies $\SEM$, and 
$\SEM + \SADS$ implies $\SRT$, we have $\EM$ implies $\SRT$, 
contradiction Theorem \ref{thm:EM}.
\end{proof}

\section{$\ADS$ does not imply $\SCAC$}
\label{sec:ADS}

\subsection{Outline}

In this section, we prove Theorem \ref{thm:ADS} by constructing a Turing ideal $\mathcal{I} \subseteq \mathcal{P}(\omega)$ such that 
$(\omega, \mathcal{I}) \vDash \ADS$ and $\mathcal{I}$ contains a stable partial order $M = (\mathbb{N}, \preceq_M)$ but does not contain a 
solution to $M$.  The construction proceeds in two steps; we use a ground forcing to build $M$ followed by an iterated forcing to add solutions 
to infinite linear orders without adding a solution to $M$.  

Recall that for an infinite poset $M$, $A^*(M)$ is the set of elements which are below almost every element and 
$B^*(M)$ is the set of elements which are incomparable with almost every element.  In the ground forcing, we specify $A^*(M)$ and 
$B^*(M)$ as we construct $M$ so that $A^*(M) \cup B^*(M) = \mathbb{N}$ and hence $M$ is stable.  We satisfy two types of requirements.  First, 
to ensure that $M$ cannot compute a solution to itself it suffices to ensure that if $\Phi_e^M$ is infinite, then $\Phi_e^M(a) = \Phi_e^M(b) = 1$ for 
some $a \in A^*(M)$ and $b \in B^*(M)$.  Since we are defining $A^*(M)$ and $B^*(M)$ as we construct $M$, these are easy to satisfy.   Second, 
we satisfy ground level requirements which guarantee that requirements for the first level of the iteration forcing are appropriately dense (in a sense 
defined below). 

For the first level of the iteration forcing, we begin with $M$, $A^*(M)$ and $B^*(M)$ already defined.  We fix an index $e$ such that $\Phi_e^M$ is an 
infinite linear order and attempt to add a solution $f$ for $\Phi_e^M$ to $\mathcal{I}$ so that $M \oplus f$ does not compute a solution to $M$.  
As above, the strategy is to show that if $\Phi_e^{M \oplus f}$ is infinite, then there are elements $a \in A^*(M)$ and $b \in B^*(M)$ such that 
$\Phi_e^{M \oplus f}(a) = \Phi_e^{M \oplus f}(b) = 1$.  However, since $A^*(M)$ and $B^*(M)$ are already defined, implementing this strategy requires 
using the fact that the ground forcing ensured that requirements for the iterated forcing are appropriately dense.  This density will mean 
that as $f$ is defined, if there are lots of options to force large numbers into $\Phi_e^{M \oplus f}$, then there must be numbers from $A^*(M)$ and 
$B^*(M)$ in $\Phi_e^{M \oplus f}$.  In addition to handling these diagonalization strategies, we need to guarantee that the 
requirements for the next level of the iteration forcing are appropriately dense.  

In the construction below, we explain the iteration forcing first (assuming $M$, $A^*(M)$ and $B^*(M)$ have already been constructed) 
because it allows us to introduce the density notions that have to be forced at the ground level.  After explaining the iteration forcing, 
we present the ground forcing to construct $M$, $A^*(M)$ and $B^*(M)$.  

Before starting the construction, we restrict the collection of infinite linear orders for which we need to add solutions to $\mathcal{I}$.  

\begin{defn}  
\label{defn:V}
A linear ordering $(\mathbb{N}, \prec)$ is \emph{stable-ish} if there is a non-empty initial segment $V$ which has no maximum under $\prec$, and such that 
$\mathbb{N}\setminus V$ is non-empty and has no minimum under $\prec$.
\end{defn}
Note that there is no requirement that the set $V$ be computable from $\prec$.

\begin{lem}
  If $(\mathbb{N},\prec)$ is not stable-ish then there is a solution to $(\mathbb{N},\prec)$ computable from $\prec$.
\end{lem}
\begin{proof}
  Assume $(\mathbb{N},\prec)$ is not stable-ish.  Note that if $V$ is a non-empty initial segment with no maximum element, then $V$ can compute 
an infinite ascending sequence.  Let $a_1\in V$ be arbitrary.  Given $a_n$, there must be infinitely many elements $x \in V$ such that 
$a_n \prec x$, so simply search (effectively in $V$) for such an element and set $a_{n+1} = x$.

If there is a non-empty initial segment $V$ with no maximum, observe that since $\prec$ is not stable-ish, either $\mathbb{N}\setminus V=\emptyset$, 
in which case $V$ is computable, or $\mathbb{N}\setminus V$ has a minimal element $b$, in which case $V=\{x\mid x\prec b\}$.  In either case, 
$V$ is computable from $\prec$, and so there is an infinite ascending sequence computable from $\prec$.

So suppose there is no such $V$.  Then every non-empty initial segment has a maximum element.  Let $V$ be the set of elements with finitely 
many predecessors.  $V$ is either empty or finite, since if $V$ were infinite, it would not have a maximal element.  Thus $\mathbb{N}\setminus V$ 
is computable from $\prec$, and can have no minimal element. (Any minimal element would have only the finitely many elements of $V$ 
as predecessors, and would therefore belong to $V$.)  Therefore, by an argument similar to the one above, $\mathbb{N} \setminus V$
contains an infinite descending sequence computable from $\prec$.  
\end{proof}

We end this subsection by fixing some notation and conventions.  If $\sigma$ and $\delta$ are finite strings, then $\sigma^\frown\delta$ denotes the concatenation of 
$\sigma$ and $\delta$.  We write $\sigma \sqsubseteq \tau$ to denote that $\sigma$ is an initial segment of $\tau$ (i.e.~$\tau = \sigma^\frown\delta$ for some 
string $\delta$).  If $\prec$ is a linear order on $\mathbb{N}$, $\sigma$ is a finite sequence which is ascending in $\prec$ and 
$\tau$ is a finite sequence which is descending in $\prec$, then $\sigma \prec \tau$ means that $\sigma(|\sigma|-1) \prec \tau(|\tau|-1)$ (i.e.~the last element in 
$\sigma$ is strictly below the last element in $\tau$ in the $\prec$ order).

For any computation in which part of the oracle is a finite string, for example $\Phi_k^{X \oplus \sigma}$, we follow the standard convention 
that if $\Phi_k^{X \oplus \sigma}(y)$ converges, then both $y$ and the use of the computation are bounded by $|\sigma|$.  

\subsection{Iteration Forcing}\label{sec:it1}

Assume that we have already used the ground 
forcing to construct our stable poset $(M, \preceq_M)$ along with $A^*(M)$ and $B^*(M)$.  The general context for one step of the iteration forcing will 
be a fixed set $X$ and an index $e$ meeting the following conditions: 
\begin{itemize}
\item $M \leq_T X$; 
\item $X$ does not compute a solution to $M$; 
\item $\Phi_e^X$ is the characteristic function for a stable-ish linear order $\prec_e^X$ on $\mathbb{N}$; and 
\item each requirement $\mathcal{K}^{X,A^*(M),B^*(M)}$ is uniformly dense (defined below).  
\end{itemize}
The ground forcing will create this context for $X = M$.  Our goal is to 
find a generic solution $G$ for $\prec_e^X$ (either an infinite ascending or descending sequence) 
such that $X \oplus G$ does not compute a solution to $M$ and such that for each stable-ish linear 
order $\prec_{e'}^{X \oplus G}$, the requirements $\mathcal{K}^{X \oplus G,A^*(M),B^*(M)}$ are uniformly dense.  We add $G$ to the Turing ideal and 
note that for any index $e'$ such that $\prec_{e'}^{X \oplus G}$ is a stable-ish linear order, 
we have created the context for the iteration forcing to continue with $X \oplus G$.

Before giving the specifics of our forcing notion, we describe the basic intuition for constructing a solution $G$ for $\prec_e^X$ while 
diagonalizing against computing a solution to $M$ from $X \oplus G$.  We work with pairs $(\sigma,\tau)$ where $\sigma$ is a finite ascending 
sequence in $\prec_e^X$, $\tau$ is a finite descending sequence in $\prec_e^X$ and $\sigma \prec_e^X \tau$.  We view this pair as a 
simultaneous attempt to build an infinite ascending solution and an infinite descending solution to $\prec_e^X$.  The goal is to construct 
an infinite nested sequence of such pairs $(\sigma_k, \tau_k)$ such that we succeed either with $G = \sigma = \cup \sigma_k$ or with 
$G = \tau = \cup \tau_k$.

Suppose we have constructed a pair $(\sigma_k,\tau_k)$.  A typical diagonalization requirement is specified by a pair of indices $m$ and $n$.  
To meet this requirement, we need to either 
\begin{itemize}
\item find an ascending sequence $\sigma_{k+1}$ extending $\sigma_k$ such that $\sigma_{k+1} \prec_e^X \tau_k$ and there 
exists a pair of elements 
$a \in A^*(M)$, $b \in B^*(M)$ such that $\Phi_m^{X \oplus \sigma_{k+1}}(a) = \Phi_m^{X \oplus \sigma_{k+1}}(b) = 1$; or
\item find a descending sequence $\tau_{k+1}$ extending $\tau_k$ such that $\sigma_k \prec_e^X \tau_{k+1}$ and there exists a pair of elements 
$a \in A^*(M)$, $b \in B^*(M)$ such that $\Phi_n^{X \oplus \tau_{k+1}}(a) = \Phi_n^{X \oplus \tau_{k+1}}(b) = 1$.
\end{itemize}
That is, we extend our approximation to an ascending solution to $\prec_e^X$ in a manner that diagonalizes or we extend our approximation to a 
descending solution to $\prec_e^X$ in a manner that diagonalizes.  If we can always win on the ascending side, then $G = \cup \sigma_k$ is an 
infinite ascending solution to $\prec_e^X$ such that $X \oplus G$ cannot compute a solution to $M$.   Otherwise, there is an index $m$ for which we 
cannot win on the ascending side.  In this case, we must win on the descending side for every index $n$ (when it is paired with $m$) and hence 
$G = \cup \tau_k$ is an appropriate infinite descending solution to $\prec_e^X$.

In general, there is no reason to think we can meet these requirements without some additional information about $X$.  It is the fact that 
each requirement $\mathcal{K}^{X,A^*(M),B^*(M)}$ is uniformly dense which allows us to meet these requirements.  We first focus on formalizing 
these diagonalization requirements in a general context and then we show why this 
general context also forces the requirements $\mathcal{K}^{X \oplus G, A^*(M),B^*(M)}$ to be uniformly dense at the next level.  

We begin by defining the following sets, each computable from $X$.
\begin{gather*}
\mathbb{A}_e^X = \{ \sigma \mid \sigma \text{ is a finite ascending sequence in} \prec_e^X \} \\
\mathbb{D}_e^X = \{ \tau \mid \tau \text{ is a finite descending sequence in} \prec_e^X \} \\
\mathbb{P}_e^X = \{ (\sigma,\tau) \mid \sigma \in \mathbb{A}_e^X \wedge \tau \in \mathbb{D}_e^X \wedge \sigma \prec_e^X \tau \}
\end{gather*}
$\mathbb{P}_e^X$ is our set of forcing conditions.  For $p \in \mathbb{P}_e^X$, we let $\sigma_p$ and $\tau_p$ denote the first and 
second components of $p$.  For $p,q \in \mathbb{P}_e^X$, we say $q \leq p$ if $\sigma_p \sqsubseteq \sigma_q$ and 
$\tau_p \sqsubseteq \tau_q$. 

To define the generic $G$, we construct a sequence $p_0 \geq p_1 \geq p_2 \geq \cdots$ of conditions $p_n = (\sigma_n,\tau_n) \in 
\mathbb{P}_e^X$.  At the $(n+1)$st step, we define $p_{n+1} \leq p_n$ to meet the highest priority requirement 
$\mathcal{K}^{X,A^*(M),B^*(M)}$ which is not yet satisfied.  Meeting this requirement will make progress either towards making 
$\sigma = \cup_n \sigma_n$ our desired infinite ascending 
solution to $\prec_e^X$ or towards making $\tau = \cup_n \tau_n$ our desired infinite descending solution to $\prec_e^X$.  In the end, we show that 
one of $G = \sigma$ or $G = \tau$ satisfies all the requirements.

Before defining the requirements, there is one obvious worry we need to address.  
During this process, we need to avoid taking a step which eliminates either side from being extendible to a solution of  
$\prec_e^X$.  Because $\prec_e^X$ is stable-ish, we fix a set $V$ for $\prec_e^X$ as in Definition \ref{defn:V}.  
We define 
\[
\mathbb{V}_e^X = \{ (\sigma, \tau) \in \mathbb{P}_e^X \mid \sigma \subseteq V \, \wedge \, \tau \subseteq \mathbb{N} \setminus V \}.
\]
For $(\sigma,\tau) \in \mathbb{V}_e^X$, $\sigma$ is an initial segment of an increasing solution to $\prec_e^X$ and 
$\tau$ is an initial segment of a decreasing solution to $\prec_e^X$.  Therefore, as long as we choose our generic sequence to lie within 
$\mathbb{V}_e^X$, we will never limit either side from being extendible to a solution to $\prec_e^X$.  However, working strictly in 
$\mathbb{V}_e^X$ has the disadvantage that $\mathbb{V}_e^X$ is not computable from $X$.  We reconcile the advantages of 
working in $\mathbb{P}_e^X$ (which is computable from $X$) with working in $\mathbb{V}_e^X$ by using split pairs.

\begin{defn}
A \emph{split pair} below $p=(\sigma_p,\tau_p)$ is a pair of conditions $q_0=(\sigma_p^\frown\sigma',\tau_p)$ and $q_1=(\sigma_p,\tau_p^\frown\tau')$ 
such that $\sigma'\prec_e^X\tau'$.  
\end{defn}

\begin{lem}
\label{split}
  If $p\in\mathbb{V}_e^X$ and $q_0,q_1$ is a split pair below $p$ then either $q_0\in\mathbb{V}_e^X$ or $q_1\in\mathbb{V}_e^X$.
\end{lem}

\begin{proof}
Let $q_0=(\sigma_p^\frown\sigma',\tau_p)$ and $q_1=(\sigma_p,\tau_p^\frown\tau')$.  Suppose $q_0\not\in\mathbb{V}$.  Since 
$\sigma_p^\frown\sigma' \prec_e^X \tau_p$, it must be that $\sigma'$ overflows from $V$ into $\mathbb{N} \setminus V$.  Therefore, since $\sigma'\prec_e^X\tau'$, 
$q_1\in\mathbb{V}$.
\end{proof}

We will use Lemma \ref{split} as follows.  Each requirement $\mathcal{K}^{X,A^*(M),B^*(M)}$ will have the property that when we need to meet 
$\mathcal{K}^{X,A^*(M),B^*(M)}$ below an element $p_n$ in our generic sequence, there will be a split pair $q_0,q_1$ (from $\mathbb{P}_e^X$) 
below $p_n$ in 
$\mathcal{K}^{X,A^*(M),B^*(M)}$.  Therefore, if $p_n \in \mathbb{V}_e^X$ by induction, then we can meet $\mathcal{K}^{X,A^*(M),B^*(M)}$ within 
$\mathbb{V}_e^X$ by choosing $p_{n+1}$ to be whichever of  $q_0$ and $q_1$ is in $\mathbb{V}_e^X$.    Thus, by starting with 
the empty sequence $p_0$ (which is in $\mathbb{V}_e^X$), we can assume that our generic sequence is chosen in $\mathbb{V}_e^X$.

We have two types of requirements: half requirements and full requirements.  For uniformity of presentation, it is easiest to deal with a general 
definition for the full requirements, although in the end, the only full requirements we need to meet are those made up of a pair of half requirements.

\begin{defn}
We define the following types of requirements and half-requirements.
\begin{itemize}
\item A \emph{requirement} is a downward closed set $\mathcal{K}^{X,A^*(M),B^*(M)} \subseteq \mathbb{P}_e^X$ such that 
\[
\mathcal{K}^{X,A^*(M),B^*(M)} = \{ p \in \mathbb{P}_e^X \mid \exists \overline{a} \in A^*(M) \, \exists \overline{b} \in B^*(M) 
\, (K^X(p,\overline{a},\overline{b})) \}
\]
for some relation $K^X(x,\overline{y},\overline{z})$ computable in $X$.  
\item An $\mathbb{A}$-\emph{side half requirement} is a set $\mathcal{R}^{X,A^*(M),B^*(M)} \subseteq \mathbb{A}_e^X$ which is closed under 
extensions such that 
\[
\mathcal{R}^{X,A^*(M),B^*(M)} = \{ \sigma \in \mathbb{A}_e^X \mid \exists \overline{a} \in A^*(M) \, \exists \overline{b} \in B^*(M) \, 
( R^X(\sigma,\overline{a},\overline{b})) \}
\]
for some relation $R^X(x,\overline{y},\overline{z})$ computable in $X$.  
\item A $\mathbb{D}$-\emph{side half requirement} is a set $\mathcal{S}^{X,A^*(M),B^*(M)} \subseteq \mathbb{D}_e^X$ which is closed under 
extensions such that 
\[
\mathcal{S}^{X,A^*(M),B^*(M)} = \{ \tau \in \mathbb{D}_e^X \mid \exists \overline{a} \in A^*(M) \, \exists \overline{b} \in B^*(M) \, 
( S^X(\tau,\overline{a},\overline{b})) \}
\]
for some relation $S^X(x,\overline{y},\overline{z})$ computable in $X$.  
\item If $\mathcal{R}^{X,A^*(M),B^*(M)}$ is an $\mathbb{A}$-side half requirement and $\mathcal{S}^{X,A^*(M),B^*(M)}$ is a $\mathbb{D}$-side half requirement, 
then $\mathcal{J}_{\mathcal{R},\mathcal{S}}^{X,A^*(M),B^*(M)}$ is the requirement 
\[
\mathcal{J}_{\mathcal{R},\mathcal{S}}^{X,A^*(M),B^*(M)} = \{ p \in \mathbb{P}_e^X \mid \sigma_p \in \mathcal{R}^{X,A^*(M),B^*(M)} \vee \tau_p \in 
\mathcal{S}^{X,A^*(M),B^*(M)} \}.
\]
\end{itemize}
\end{defn}

We say $\mathcal{R}^{X,A^*(M),B^*(M)}$ is a half requirement to mean that it is either an $\mathbb{A}$-side or a $\mathbb{D}$-side 
half requirement.  Each requirement and half requirement is c.e.~in $X \oplus A^*(M) \oplus B^*(M)$ and 
the dependence on $A^*(M)$ and $B^*(M)$ is positive.  

\begin{example}
\label{mainexample}
Fix a pair of indices $m$ and $n$.  The formal version of our basic diagonalization strategy is given by the following half requirements:
\begin{gather*}
\mathcal{A}_m^{X,A^*(M),B^*(M)} = \{ \sigma \in \mathbb{A}_e^X \mid \exists a \in A^*(M) \, \exists b \in B^*(M) \, 
(\Phi_{m}^{X \oplus \sigma}(a) = \Phi_{m}^{X \oplus \sigma}(b) = 1) \}, \\
\mathcal{D}_n^{X,A^*(M),B^*(M)} = \{ \tau \in \mathbb{D}_e^X \mid \exists a \in A^*(M) \, \exists b \in B^*(M) \, 
(\Phi_{n}^{X \oplus \tau}(a) = \Phi_{n}^{X \oplus \tau}(b) = 1) \}.
\end{gather*}
These half requirements combine to form the requirement 
\[
\mathcal{J}_{\mathcal{A}_m,\mathcal{D}_n}^{X,A^*(M),B^*(M)} = 
\left\{ p \in \mathbb{P}_e^X \mid \sigma_p \in \mathcal{A}_m^{X,A^*(M),B^*(M)} \vee \tau_p \in \mathcal{D}_n^{X,A^*(M),B^*(M)} \right\}.
\]
Notice that if $\sigma \in \mathcal{A}_m^{X,A^*(M),B^*(M)}$ and $\sigma \sqsubseteq G$, then $\Phi_m^{X \oplus G}$ is not a solution to 
$M$.  Similarly, if $\tau \in \mathcal{D}_n^{X,A^*(M),B^*(M)}$ and $\tau \sqsubseteq G$, then $\Phi_n^{X \oplus G}$ is not a solution to $M$.  
\end{example}

We next describe when an $\mathbb{A}$-side half requirement $\mathcal{R}^{X,A^*(M),B^*(M)}$ is satisfied by an infinite ascending  
sequence $\Lambda$ in $\prec_e^X$.  (With the obvious changes, this description applies to a $\mathbb{D}$-side half requirement 
$\mathcal{S}^{X,A^*(M),B^*(M)}$ and an infinite descending sequence $\Lambda$.)  
$\mathcal{R}^{X,A^*(M),B^*(M)}$ is specified by an index $i$ such that 
\[
\mathcal{R}^{X,A^*(M),B^*(M)} = \{ \sigma \in \mathbb{A}_e^X \mid \exists \overline{a} \in A^*(M) \, \exists \overline{b} \in B^*(M) \, 
( \Phi_i^X(\sigma,\overline{a},\overline{b}) = 1) \}
\]
where $\Phi_i^X$ is total.  For any (typically finite) sets $A$ and $B$ (given by canonical indices), we let 
\[
\mathcal{R}^{X,A,B} = \{ \sigma \in \mathbb{A}_e^X \mid \exists \overline{a} \in A \, \exists \overline{b} \in B \, 
( \Phi_i^X(\sigma,\overline{a},\overline{b}) = 1) \}.
\]
Unlike $\mathcal{R}^{X,A^*(M),B^*(M)}$, the set $\mathcal{R}^{X,A,B}$ is not necessarily 
closed under extensions.  However, for any finite sets $A$ and $B$, we have $\mathcal{R}^{X,A,B} \leq_T X$.  

 We write $\mathcal{R}^X$ to indicate the operation mapping $A,B$ to $\mathcal{R}^{X,A,B}$.  

\begin{defn}
\label{def:halfessential}
$\mathcal{R}^{X}$ is \emph{essential} in $\Lambda$ if for every $n$ and every $x$, there is a finite set $A > x$ such that 
for every $y$, there is a finite set $B > y$ and an $m>n$ so that $\Lambda\upharpoonright m\in\mathcal{R}^{X,A,B}$.  
We say the infinite ascending sequence $\Lambda$ \emph{satisfies} $\mathcal{R}^{X,A^*(M),B^*(M)}$ if either
  \begin{itemize}
  \item $\mathcal{R}^{X}$ is not essential in $\Lambda$, or
  \item there is an $n$ such that $\Lambda\upharpoonright n\in\mathcal{R}^{X,A^*(M),B^*(M)}$.
  \end{itemize}
\end{defn}

\begin{example}
Consider the $\mathbb{A}$-side diagonalization half requirement 
\[
\mathcal{A}_m^{X,A^*(M),B^*(M)} = \{ \sigma \in \mathbb{A}_e^X \mid \exists a \in A^*(M) \, \exists b \in B^*(M) \, 
(\Phi_{m}^{X \oplus \sigma}(a) = \Phi_{m}^{X \oplus \sigma}(b) = 1) \}
\]
and an infinite ascending sequence $\Lambda$ in $\prec_e^X$.  $\mathcal{A}_m^{X}$ is essential in $\Lambda$ if and only if 
$\Phi_m^{X \oplus \Lambda}$ is infinite.  
Therefore, $\mathcal{A}_m^{X,A^*(M),B^*(M)}$ is satisfied by $\Lambda$ if and only if either $\Phi_m^{X \oplus \Lambda}$ is finite 
or there exists $a \in A^*(M)$ and $b \in B^*(M)$ 
such that $\Phi_m^{X \oplus \Lambda}(a) = \Phi_m^{X \oplus \Lambda}(b) = 1$.  In either case, $\Lambda$ is a solution to $\prec_e^X$ such that  
$\Phi_m^{X \oplus \Lambda}$ is not a solution to $M$.  
\end{example}

This example does not explain why we need the quantifier alternations in Definition \ref{def:halfessential}.  This quantifier alternation will be 
reflected in a similar definition for full requirements and the reason for it will become clear in the ground forcing.  

We need similar notions in the context of our (full) requirements.  Each requirement $\mathcal{K}^{X,A^*(M),B^*(M)}$ is specified by 
an index $i$ such that 
\[
\mathcal{K}^{X,A^*(M),B^*(M)} = \{ p \in \mathbb{P}_e^X \mid \exists \overline{a} \in A^*(M) \, \exists \overline{b} \in B^*(M) \, 
( \Phi_i^X(p,\overline{a},\overline{b}) = 1) \}
\]
where $\Phi_i^X$ is total.  For any (typically finite) sets $A$ and $B$, we let 
\[
\mathcal{K}^{X,A,B} = \{ p \in \mathbb{P}_e^X \mid \exists \overline{a} \in A \, \exists \overline{b} \in B \, (\Phi_i^X(p,\overline{a},\overline{b}) = 1) \}.
\]
As above, the set $\mathcal{K}^{X,A,B}$ need not be downward closed in $\mathbb{P}_e^X$, but is computable from $X$ when $A$ and $B$ 
are finite.  

\begin{defn}
$\mathcal{K}^{X}$ is \emph{essential} below $p\in\mathbb{P}_e^X$ if for every $x$, there is a finite set 
$A > x$ such that for every $y$, there is a finite set $B > y$ and a split pair $q_0,q_1$ below $p$ such that $q_0,q_1\in\mathcal{K}^{X,A,B}$.

$\mathcal{K}^{X,A^*(M),B^*(M)}$ is \emph{uniformly dense} if whenever $\mathcal{K}^{X}$ is essential below $p$, there is a split pair $q_0,q_1$ below $p$ belonging to $\mathcal{K}^{X,A^*(M),B^*(M)}$.
\end{defn}

\begin{example}
Let $\mathcal{J}_{\mathcal{A}_m, \mathcal{D}_n}^{X,A^*(M),B^*(M)}$ be the requirement from Example \ref{mainexample} and fix a condition 
$p = (\sigma_p,\tau_p)$.   Let $q_0 = (\sigma_p^\frown\sigma,\tau_p)$ and $q_1 = (\sigma_p,\tau_p^\frown\tau)$ be a split pair below $p$.  
For finite sets $A$ and $B$, $q_0 \in \mathcal{J}_{\mathcal{A}_m, \mathcal{D}_n}^{X,A,B}$ if
\[
\exists a \in A \, \exists b \in B \, \big( \Phi_m^{X \oplus \sigma_p^\frown\sigma}(a) =  \Phi_m^{X \oplus \sigma_p^\frown\sigma}(b) = 1 \vee 
\Phi_n^{X \oplus \tau_p}(a) =  \Phi_n^{X \oplus \tau_p}(b) = 1 \big).
\]
For $A > |\tau_p|$, the second disjunct cannot occur by our use convention, and hence 
\[
q_0 \in \mathcal{J}_{\mathcal{A}_m, \mathcal{D}_n}^{X,A,B} \, \Leftrightarrow \, 
\exists a \in A \, \exists b \in B \, \big( \Phi_m^{X \oplus \sigma_p^\frown\sigma}(a) =  \Phi_m^{X \oplus \sigma_p^\frown\sigma}(b) = 1\big).
\]
Similarly, if $B > |\sigma_p|$, then 
\[
q_1 \in \mathcal{J}_{\mathcal{A}_m, \mathcal{D}_n}^{X,A,B} \, \Leftrightarrow \, 
\exists a \in A \, \exists b \in B \, \big( \Phi_n^{X \oplus \tau_p^\frown\tau}(a) =  \Phi_n^{X \oplus \tau_p^\frown\tau}(b) = 1\big).
\]
Thus the definition of $\mathcal{J}_{\mathcal{A}_m, \mathcal{D}_n}^{X}$ being essential below $p$ formalizes a 
notion of ``having lots of options to force large numbers into a potential solution to $M$".  Informally, the definition of 
$\mathcal{J}_{\mathcal{A}_m, \mathcal{D}_n}^{X,A^*(M),B^*(M)}$ being uniformly dense says that whenever there are lots of options to force large numbers into 
a potential solution to $M$, then there is an extension which forces numbers from both $A^*(M)$ and $B^*(M)$ into the potential solution.
\end{example}

\begin{defn}
We say an infinite sequence $p_0>p_1>\cdots$ of conditions \emph{satisfies} $\mathcal{K}^{X,A^*(M),B^*(M)}$ if either
\begin{itemize}
\item there are cofinitely many $p_i$ such that $\mathcal{K}^{X}$ is not essential below $p_i$, or
\item there is some $p_n\in\mathcal{K}^{X,A^*(M),B^*(M)}$.
\end{itemize}
\end{defn}

We have now made all the inductive hypotheses on $X$ precise and can give the formal construction of our generic 
sequence of conditions.  Let $\mathcal{K}_n^{X,A^*(M),B^*(M)}$, for $n \in \omega$, be a list of all requirements.  (As we will see below, it suffices for 
this list to consist of all requirements formed from pairs of half requirements.)

\begin{lem}
\label{lem:genericsequence}
There is a sequence of conditions $p_0 > p_1 > \cdots$ from $\mathbb{V}_e^X$ which satisfies every $\mathcal{K}_n^{X,A^*(M),B^*(M)}$.
\end{lem} 

\begin{proof}
Let $p_0=(\sigma_0, \tau_0)$ where both $\sigma_0$ and $\tau_0$ are the empty sequence and note that $p_0 \in\mathbb{V}_e^X$.  
Given $p_n$, let $m$ be the least index such that $\mathcal{K}^{X}_m$ is essential 
below $p_n$ and for all $i \leq n$, $p_i \not \in \mathcal{K}^{X,A^*(M),B^*(M)}_m$.  By assumption $\mathcal{K}^{X,A^*(M),B^*(M)}_m$ is 
uniformly dense, so we may apply Lemma \ref{split} to obtain $p_{n+1}\leq p_n$ such that $p_{n+1} \in \mathcal{K}^{X,A^*(M),B^*(M)}_m$ and 
$p_{n+1} \in \mathbb{V}_e^X$.
\end{proof}

It remains to show that for either $G = \sigma = \cup \sigma_n$ or $G = \tau = \cup \tau_n$, $G$ satisfies the necessary inductive conditions: 
$X \oplus G$ does not compute a solution to 
$M$ and all requirements $\mathcal{K}^{X \oplus G, A^*(M),B^*(M)}$ are uniformly dense.  We do this in two steps.  First we explain the 
connection between satisfying half requirements and satisfying full requirements.  Second, we show that the satisfaction of the appropriate half 
requirements forces these conditions for $X \oplus G$.  

\begin{lem}
\label{lem:halftofull}
Let $\mathcal{R}^{X,A^*(M),B^*(M)}$ and $\mathcal{S}^{X,A^*(M),B^*(M)}$ be half-requirements and $p_0 > p_1 > \cdots$ 
be an infinite sequence of conditions with $p_n = (\sigma_n, \tau_n)$.  Let $\sigma=\bigcup_i\sigma_i$, $\tau=\bigcup_i\tau_i$.  
If $\mathcal{R}^{X}$ is 
essential in $\sigma$ and $\mathcal{S}^{X}$ is essential in $\tau$, then 
$\mathcal{J}^{X}_{\mathcal{R},\mathcal{S}}$ is essential below every $p_n$.
\end{lem}

\begin{proof}
Fix $p_n$.  To show $\mathcal{J}^{X}_{\mathcal{R},\mathcal{S}}$ is essential below every $p_n$, fix $x$.   Let $A_0 > x$ witness 
that $\mathcal{R}^{X}$ is essential in $\sigma$ and let $A_1 > x$ witness that $\mathcal{S}^{X}$ is essential in 
$\tau$.  $A_0 \cup A_1$ will be our witness that $\mathcal{J}^{X}_{\mathcal{R},\mathcal{S}}$ is essential below $p_n$.  

Fix $y$.  Let $B_0 > y$ witness that $\mathcal{R}^{X}$ is essential in $\sigma$ and let $B_1 > y$ witness that 
$\mathcal{S}^{X}$ is essential in $\tau$.  $B_0 \cup B_1$ will be our witness that 
$\mathcal{J}^{X}_{\mathcal{R},\mathcal{S}}$ is essential below $p_n$.

Fix $m_0 > n$ such that $\sigma_{m_0} \in \mathcal{R}^{X,A_0,B_0}$ and fix $m_1 > n$ such that $\tau_{m_1} \in \mathcal{S}^{X,A_1, B_1}$.  
Because the dependence on $A_0$, $A_1$, $B_0$ and $B_1$ in these sets is positive, it follows that 
$\sigma_{m_0} \in \mathcal{R}^{X,A_0 \cup A_1, B_0 \cup B_1}$ and 
$\tau_{m_1} \in \mathcal{S}^{X,A_0 \cup A_1, B_0 \cup B_1}$.  Thus the conditions $(\sigma_{m_0}, \tau_n)$ and $(\sigma_n, \tau_{m_1})$ 
are in $\mathcal{J}_{\mathcal{R},\mathcal{S}}^{X,A_0 \cup A_1, B_0 \cup B_1}$ and form a split pair below $p_n$.  
\end{proof}

Putting these pieces together, we obtain the following:
\begin{lem}\label{generic_iteration}
Suppose that for each pair of half-requirements $\mathcal{R}^{X,A^*(M),B^*(M)}$ and $\mathcal{S}^{X,A^*(M),B^*(M)}$, the 
requirement $\mathcal{J}^{X,A^*(M),B^*(M)}_{\mathcal{R},\mathcal{S}}$ is uniformly dense.  Then there is an infinite sequence $(\sigma_0,\tau_0)>(\sigma_1,\tau_1)>\cdots$ of conditions such that, setting $\sigma=\bigcup_i\sigma_i$ and $\tau=\bigcup_i\tau_i$, 
either $\sigma$ satisfies every $\mathbb{A}$-side half-requirement or $\tau$ satisfies every $\mathbb{D}$-side half-requirement.
\end{lem}

\begin{proof}
Let $p_0 > p_1 > \cdots$ be chosen as in Lemma \ref{lem:genericsequence}.  Since each $\mathcal{J}^{X,A^*(M),B^*(M)}_{\mathcal{R},\mathcal{S}}$ 
is uniformly dense, this sequence satisfies every requirement $\mathcal{J}^{X,A^*(M),B^*(M)}_{\mathcal{R},\mathcal{S}}$.  If $\sigma$ satisfies 
every half-requirement, we are done.  So suppose there is some $\mathcal{R}^{X,A^*(M),B^*(M)}$ not satisfied by $\sigma$, and note that 
$\mathcal{R}^{X}$ must be essential in $\sigma$.  We show that $\tau$ satisfies every $\mathcal{S}^{X,A^*(M),B^*(M)}$.  

Fix $\mathcal{S}^{X,A^*(M),B^*(M)}$ and assume that $\mathcal{S}^{X}$ is essential in $\tau$ (otherwise this half requirement is trivially satisfied). By Lemma \ref{lem:halftofull}, $\mathcal{J}^{X}_{\mathcal{R},\mathcal{S}}$ is essential for every 
$(\sigma_n,\tau_n)$, and since the sequence of conditions satisfies $\mathcal{J}^{X,A^*(M),B^*(M)}_{\mathcal{R},\mathcal{S}}$, there must be 
some condition $(\sigma_n,\tau_n)\in\mathcal{J}^{X,A^*(M),B^*(M)}_{\mathcal{R},\mathcal{S}}$.  We cannot have $\sigma_n\in\mathcal{R}^{X,A^*(M),B^*(M)}$, since then $\sigma$ would satisfy $\mathcal{R}^{X,A^*(M),B^*(M)}$, so $\tau_n\in\mathcal{S}^{X,A^*(M),B^*(M)}$.
\end{proof}

We set $G = \sigma$ if $\sigma$ satisfies all the $\mathbb{A}$-side half requirements and we set $G = \tau$ otherwise.  
By Lemma \ref{generic_iteration}, $G$ satisfies every half requirement (on the appropriate side).  
It remains to show that $X \oplus G$ does not compute a solution to $M$ and that each requirement 
$\mathcal{K}^{X \oplus G, A^*(M), B^*(M)}$ is uniformly dense.  We work under the hypothesis that $G = \sigma$ and hence restrict our attention to 
$\mathbb{A}$-side half requirements.  The same arguments, with the obvious changes, give the corresponding results if $G = \tau$ working with 
$\mathbb{D}$-side half requirements.

\begin{lem}
\label{lem:nosolution}
If $G$ satisfies every $\mathcal{A}_m^{X,A^*(M),B^*(M)}$ half requirement, then $X \oplus G$ does not compute a solution to $M$.
\end{lem}

\begin{proof}
Fix an index $m$.  If $\Phi_m^{X \oplus G}$ is finite, then we are done.  So, suppose $\Phi_m^{X \oplus G}$ is infinite.  We claim that 
$\mathcal{A}_m^{X}$ is essential in $G$.  To prove this claim, fix $n$ and $x$.  Let $a_0 > x$ be such that 
$\Phi_m^{X \oplus G}(a_0) = 1$ and set $A = \{ a_0 \}$.  Fix $y$, let $b_0 > y$ be such that $\Phi_m^{X \oplus G}(b_0) = 1$ and 
set $B = \{ b_0 \}$.  Set $n' > n$ be greater than the use of either of these computations.  
By definition, $G \upharpoonright n' \in \mathcal{A}_m^{X,A,B}$ and hence $\mathcal{A}_m^{X}$ is essential in $G$.  

Since $G$ satisfies $\mathcal{A}_m^{X,A^*(M),B^*(M)}$, there must be an $n''$ such that $G \upharpoonright n'' \in 
\mathcal{A}_m^{X,A^*(M),B^*(M)}$.  Therefore, for some $a \in A^*(M)$ and $b \in B^*(M)$, we have 
$\Phi_m^{X \oplus G}(a) = \Phi_m^{X \oplus G}(b) = 1$, completing the proof.
\end{proof}

Finally, we show that for every index $e'$ such that $\prec_{e'}^{X \oplus G}$ is a stable-ish linear order, each requirement 
$\mathcal{K}^{X \oplus G, A^*(M), B^*(M)} \subseteq \mathbb{P}_{e'}^{X \oplus G}$ is uniformly dense.    
Recall that $\mathcal{K}^{X \oplus G, A^*(M), B^*(M)}$ is specified by an index $i$ such that 
\[
\mathcal{K}^{X \oplus G,A^*(M),B^*(M)} = \{ p \in \mathbb{P}_{e'}^{X \oplus G} \mid \exists \overline{a} \in A^*(M) \, \exists \overline{b} \in B^*(M) \, 
( \Phi_i^{X \oplus G}(p,\overline{a},\overline{b}) = 1) \}
\]
where $\Phi_i^{X \oplus G}$ is total.  As we construct $G$, we do not know which indices $e'$ will result in $\prec_{e'}^{X \oplus G}$ being a 
stable-ish linear order and, for each such index $e'$, which indices $i$ will correspond to requirements 
$\mathcal{K}^{X \oplus G, A^*(M), B^*(M)} \subseteq \mathbb{P}_{e'}^{X \oplus G}$.  Therefore, we define the following $\mathbb{A}$-side 
half requirements for every pair of indices $e'$ and $i$.  (Of course, we also define the corresponding $\mathbb{D}$-side half requirements and all 
proofs that follow work equally well on the $\mathbb{D}$-side.)

\begin{defn}
Fix $\sigma \in \mathbb{A}_e^X$ and an index $e'$.  For a pair of finite strings $q = (\sigma_q,\tau_q)$, we say $q \in \mathbb{P}_{e'}^{X \oplus \sigma}$ 
if for all $i < j < |\sigma_q|$, $\sigma_q(i) \prec_{e'}^{X \oplus \sigma} \sigma_q(j)$, 
for all $i < j < |\tau_q|$, $\tau_q(j) \prec_{e'}^{X \oplus \sigma} \tau_q(i)$ and 
$\sigma_p(|\sigma_p|-1) \prec_{e'}^{X \oplus \sigma} \tau_p(|\tau_p|-1)$.  

We say $\sigma$ \emph{forces} $q \not \in \mathbb{P}_{e'}^{X \oplus G}$ if either there are $i < j < |\sigma_q|$ such that 
$\sigma_q(j) \preceq_{e'}^{X \oplus \sigma} \sigma_q(i)$ or there are $i < j < |\tau_q|$ such that $\tau_q(i) \preceq_{e'}^{X \oplus \sigma} 
\tau_q(j)$ or $\tau_p(|\tau_p|-1) \preceq_{e'}^{X \oplus \sigma} \sigma_p(|\sigma_p|-1)$.  Note that this definition does not match the usual 
method for forcing the negation of a statement.
\end{defn}

By the use convention, $\mathbb{P}_{e'}^{X \oplus \sigma}$ is finite and we can $X$-computably quantify over this finite set.  Furthermore, we can 
$X$-computably determine whether $\sigma$ forces $q \not \in \mathbb{P}_{e'}^{X \oplus G}$.  

\begin{defn}
For each pair of indices $e'$ and $i$ and each $q = (\sigma_q, \tau_q)$, we define the $\mathbb{A}$-side half requirement 
$\mathcal{T}_{e',i,q}^{X,A^*(M),B^*(M)}$ to be the set of all $\sigma \in \mathbb{A}_e^X$ such that either 
$\sigma$ forces $q \not \in \mathbb{P}_{e'}^{X \oplus G}$ or there exist strings $\sigma'$ and $\tau'$ such that 
$q_0 = (\sigma_q^\frown \sigma', \tau_q)$ and $q_1 = (\sigma_q, \tau_q^\frown \tau')$ satisfy 
$q_0, q_1 \in \mathbb{P}_{e'}^{X \oplus \sigma}$ and 
\[
\exists \overline{a}_0, \overline{a}_1 \in A^*(M) \, \exists \overline{b}_0, \overline{b}_1 \in B^*(M) \, 
(\Phi_{i}^{X \oplus \sigma}(q_0,\overline{a}_0,\overline{b}_0) = \Phi_{i}^{X \oplus \sigma}(q_1,\overline{a}_1,\overline{b}_1) = 1)
\]
(i.e.~$\sigma$ forces the existence of a split pair below $q$ which lies in $\mathcal{K}^{X,A^*(M),B^*(M)}$).  
\end{defn}

Let $G$ be the generic constructed by 
our iterated forcing as in Lemma \ref{generic_iteration} and assume $G = \sigma$.  Thus, $G$ satisfies every $\mathbb{A}$-side half requirement 
$\mathcal{T}_{e',i,q}^{X,A^*(M),B^*(M)}$.

Fix an index $e'$ such that $\prec_{e'}^{X \oplus G}$ is a 
stable-ish linear order and fix an index $i$ specifying a requirement 
\[
\mathcal{K}^{X \oplus G,A^*(M),B^*(M)} = \{ q \in \mathbb{P}_{e'}^{X \oplus G} \mid \exists \overline{a} \in A^*(M) \, \exists \overline{b} \in B^*(M) \, 
(\Phi_i^{X \oplus G}(q,\overline{a},\overline{b}) = 1) \}
\]
The following lemma (and its $\mathbb{D}$-side counterpart) complete our verification of the properties of the iteration forcing.

\begin{lem}
\label{lem:uniformdense}
If $G$ satisfies $\mathcal{T}_{e',i,q}^{X,A^*(M),B^*(M)}$ for every $q$, then $\mathcal{K}^{X \oplus G,A^*(M),B^*(M)}$ is uniformly dense in 
$\mathbb{P}_{e'}^{X \oplus G}$.
\end{lem}

\begin{proof}
Fix $q \in \mathbb{P}_{e'}^{X \oplus G}$ and assume that $\mathcal{K}^{X \oplus G}$ is essential below $q$.  We claim that 
$\mathcal{T}_{e',i,q}^{X}$ is essential in $G$.  Before proving the claim, notice that this claim suffices to prove the lemma.  
Since $G$ satisfies $\mathcal{T}_{e',i,q}^{X,A^*(M),B^*(M)}$ and $\mathcal{T}_{e',i,q}^{X}$ is essential in $G$, there is an $n$ 
such that $G \upharpoonright n \in \mathcal{T}_{e',i,q}^{X,A^*(M),B^*(M)}$.  By the definition of $\mathcal{T}_{e',i,q}^{X,A^*(M),B^*(M)}$, 
since $q \in \mathbb{P}_{e'}^{X \oplus G}$, there must be 
a split pair $q_0,q_1 \in \mathbb{P}_{e'}^{X \oplus G \upharpoonright n}$ below $q$ and $\overline{a}_0, \overline{a}_1 \in A^*(M)$ and 
$\overline{b}_0, \overline{b}_1 \in B^*(M)$ such that $\Phi_i^{X \oplus G \upharpoonright n}(q_0,\overline{a}_0,\overline{b}_0) = 
\Phi_i^{X \oplus G \upharpoonright n}(q_1,\overline{a}_1,\overline{b}_1) = 1$.  Thus $q_0,q_1$ give the desired split pair 
below $q$ in $\mathcal{K}^{X \oplus G, A^*(M),B^*(M)}$.  

It remains to prove the claim that $\mathcal{T}_{e',i,q}^{X}$ is essential in $G$.  Fix $n$ and $x$.  Fix $A > x$ witnessing that 
$\mathcal{K}^{X \oplus G}$ is essential below $q$.  Fix $y$ and let $B > y$ and the split pair $q_0, q_1$ below $q$ be 
such that $q_0,q_1 \in \mathcal{K}^{X \oplus G, A, B}$.  Thus, 
\[
\exists \overline{a}_0, \overline{a}_1 \in A \, \exists \overline{b}_0, \overline{b}_1 \in B \, ( \Phi_i^{X \oplus G}(q_0, \overline{a}_0, \overline{b}_0) 
= \Phi_i^{X \oplus G}(q_1,\overline{a}_1, \overline{b}_1) = 1).
\]
Let $m > n$ be such that $m$ is greater than the uses of these computations and such that 
$q, q_0, q_1 \in \mathbb{P}_{e'}^{X \oplus G \upharpoonright m}$.  Then we have $G \upharpoonright m \in \mathcal{T}_{e',i,q}^{X,A,B}$ 
as required.
\end{proof}

\subsection{Ground Forcing}

In this section, we define the ground forcing to build $(M,A^*(M),B^*(M))$ such that $M$ does not compute a solution to itself (i.e.~it does not compute an 
infinite subset of $A^*(M)$ or $B^*(M)$) and each requirement $\mathcal{K}^{M,A^*(M),B^*(M)}$ is uniformly dense.      

Our ground forcing conditions $\mathcal{F}$ are triples $(F,A^*,B^*)$ satisfying
\begin{itemize}
\item $F$ is a finite partial order such that $\text{dom}(F)$ is an initial segment of $\omega$ and for all $x,y \in \text{dom}(F)$, $x \prec_F y$ implies 
$x < y$, and
\item $A^*\cup B^*\subseteq \text{dom}(F)$, $A^*$ is downwards closed under $\prec_F$, $B^*$ is upwards closed under 
$\prec_F$ and $A^* \cap B^* = \emptyset$.  
\end{itemize}
We say $(F,A^*,B^*) \leq (F_0,A^*_0,B^*_0)$ if:
\begin{itemize}
  \item $F$ extends $F_0$ as a partial order (i.e.~$\text{dom}(F_0) \subseteq \text{dom}(F)$ and for all $x,y \in \text{dom}(F_0)$, 
  $x \preceq_{F_0} y$ if and only if $x \preceq_F y$),
  \item $A^*_0\subseteq A^*$,
  \item $B^*_0\subseteq B^*$,
  \item whenever $a\in A^*_0$ and $x \in \text{dom}(F) \setminus \text{dom}(F_0)$, $a \preceq_F x$,
  \item whenever $b\in B^*_0$ and $x \in \text{dom}(F) \setminus \text{dom}(F_0)$, $b \not \preceq_F x$ (which implies $x$ is incomparable with $b$ since 
  $b < x$ and hence $x \not \preceq_F b$).
\end{itemize}

In what follows, we will typically write $x \in M$ rather than $x \in \text{dom}(M)$.  
We define a generic sequence of conditions $(F_0,A^*_0,B^*_0) > (F_1, A^*_1,B^*_1) > \cdots$ and let $M = \cup F_n$.  We need to 
satisfy the following properties:
\begin{enumerate}
\item[(C1)] For all $i$, there is an $n$ such that $i \in A^*_n \cup B^*_n$.  (Together with the definitions of our conditions and extensions of conditions, 
this property guarantees that $A^*(M) = \cup A^*_n$ and $B^*(M) = \cup B^*_n$ and that $M$ is stable.)
\item[(C2)] For all $e$, if $\Phi_e^M$ is infinite, then there are $a \in A^*(M)$ and $b \in B^*(M)$ such that $\Phi_e^M(a) = \Phi_e^M(b) = 1$.
\item[(C3)]  If $\prec_e^M$ is a stable-ish linear order and $\mathcal{K}^{M,A^*(M),B^*(M)} \subseteq \mathbb{P}_e^M$ is a requirement 
(as defined in the previous section), then for all $p \in \mathbb{P}_e^M$, either $\mathcal{K}^{M}$ is not essential below $p$ 
or there is a split pair $q_0,q_1$ below $p$ in $\mathcal{K}^{M,A^*(M),B^*(M)}$.
\end{enumerate}

The next three lemmas show that the appropriate set of conditions forcing these properties are dense.  For (C1), we use the following lemma.

\begin{lem} 
The set of $(F,A^*,B^*)$ such that $i\in A^*\cup B^*$ is dense in $\mathcal{F}$.  
\end{lem}

\begin{proof}
Fix $(F,A^*,B^*)$ and $i \in \omega$.  Without loss of generality, we assume $i \in F$.  
If $i \not \in A^*$, then $i \not \preceq_F a$ for all $a \in A^*$ by the downwards closure of $A^*$.   
Let $F_0 = F$, $A_0^* = A^*$ and $B_0^* = B^* \cup \{ c \in F \mid i \preceq_F c \}$.  Then $i \in B_0^*$ and $(F_0,A_0^*,B_0^*)$ 
extends $(F,A^*,B^*)$.   
\end{proof}

For (C2), we use the following standard forcing definitions (with $G$ denoting the generic variable).  We say $F \Vdash \Phi_e^G \text{ is finite}$ if 
\[
\exists k \, \forall (F_0,A_0^*,B_0^*) \leq (F,A^*,B^*) \, \forall x \, ( \Phi_e^{F_0}(x) = 1 \rightarrow x \leq k).
\]
We say $F \Vdash \Phi_e^G \not \subseteq A^*(G) \wedge \Phi_e^G \not \subseteq B^*(G)$ if 
\[
\exists a \in A^* \, \exists b \in B^* \, (\Phi_e^F(a) = 1 \wedge \Phi_e^F(b) = 1).
\]

\begin{lem}
For each index $e$, the set of conditions which either force $\Phi_e^G$ is finite or force $\Phi_e^G \not \subseteq A^*(G) \wedge \Phi_e^G \not 
\subseteq B^*(G)$ is dense in $\mathcal{F}$.  
\end{lem}

\begin{proof}
Fix $e$ and $(F,A^*,B^*)$ and assume that $(F,A^*,B^*)$ has no extension forcing $\Phi_e^G$ is finite.  Fix $x > F$ and an extension 
$(F_0,A_0^*,B_0^*) \leq (F,A^*,B^*)$ such that $\Phi_e^{F_0}(x) = 1$.  Without loss of generality, we can assume that $A_0^* = A^*$ and 
$B_0^* = B^*$, so $x \not \in A_0^* \cup B_0^*$.  By the definition of extensions, we know $b \not \preceq_{F_0} x$ for all $b \in B_0^*$.  Therefore, 
the condition $(F_1,A_1^*,B_1^*)$ defined by $F_1 = F_0$, $A_1^* = A_0^* \cup \{ c \in F_0 \mid c \preceq_{F_0} x \}$ and 
$B_1^* = B_0^*$ is an extension of $(F,A^*,B^*)$ such that $x \in A_1^*$ and $\Phi_e^{F_1}(a) = 1$.  

Since $(F_1,A_1^*,B_1^*)$ does not force $\Phi_e^G$ is finite, we can repeat this idea.  Fix $y > F_0$ and an extension 
$(F_2,A_2^*,B_2^*) \leq (F_1,A_1^*,B_1^*)$ such that $\Phi_e^{F_2}(y) = 0$.  Again, without loss of generality, we can assume that 
$A_2^* = A_1^*$ and $B_2^* = B_1^*$, and hence that $y \not \preceq_{F_2} a$ for any $a \in A_2^*$.  The condition $(F_3,A_3^*,B_3^*)$ 
defined by $F_3 = F_2$, $A_3^* = A_2^*$ and $B_3^* = B_2^* \cup \{ c \in F_2 \mid y \preceq_{F_2} c \}$ is an extension of 
$(F,A^*,B^*)$ forcing $\Phi_e^G \not \subseteq A^*(G) \wedge \Phi_e^G \not \subseteq B^*(G)$.   
\end{proof}

We turn to (C3).  Fix an index $e$ for a potential stable-ish linear order $\prec_e^G$.  For $p = (\sigma,\tau)$, we say 
$(F,A^*,B^*) \Vdash p \in \mathbb{P}_e^G$ if $\sigma$ is a $\prec_e^F$ ascending sequence, $\tau$ is a $\prec_e^F$ descending 
sequence and $\sigma \prec_e^F \tau$.  We say $(F,A^*,B^*) \Vdash p \not \in \mathbb{P}_e^G$ if no extension of 
$(F,A^*,B^*)$ forces $p \in \mathbb{P}_e^G$.  Obviously, the set of conditions which either force $p \in \mathbb{P}_e^G$ or force 
$p \not \in \mathbb{P}_e^G$ is dense.  

Along with the index $e$, fix an index $i$ for a potential requirement $\mathcal{K}^{G,A^*(G),B^*(G)} \subseteq \mathbb{P}_e^G$.  That is, 
we want to consider the potential requirement 
\[
\{ q \in \mathbb{P}_e^G \mid \exists \overline{a} \in A^*(G) \, \exists \overline{b} \in B^*(G) \, ( \Phi_i^{G}(q,\overline{a},\overline{b}) = 1) \}.
\]
Suppose $(F,A^*,B^*) \Vdash p \in \mathbb{P}_e^G$ for $p = (\sigma,\tau)$.  We say 
\[
(F,A^*,B^*) \Vdash \text{ there is a split pair } q_0,q_1 \text{ below } p \text{ in } \mathcal{K}^{G,A^*(G),B^*(G)}
\]
if there are $\sigma'$ and $\tau'$ such that for $q_0 = (\sigma^\frown \sigma',\tau)$ and $q_1 = (\sigma, \tau^\frown \tau')$ we have 
\begin{itemize}
\item $(F,A^*,B^*) \Vdash q_0,q_1 \in \mathbb{P}_e^G$
\item $\sigma^\frown \sigma' \prec_e^F \tau^\frown \tau'$ and 
\item $\exists \overline{a}_0, \overline{a}_1 \in A^* \, \exists \overline{b}_0, \overline{b}_1 \in B^* \, (\Phi_i^F(q_0,\overline{a}_0,\overline{b}_0) = 
\Phi_i^F(q_1, \overline{a}_1, \overline{b}_1) = 1)$. 
\end{itemize} 
Finally, we say that 
\[
(F,A^*,B^*) \Vdash \mathcal{K}^{G} \text{ is not essential below } p
\]
if for any stable partial order $(\tilde{M}, A^*(\tilde{M}), B^*(\tilde{M}))$ with $\text{dom}(\tilde{M}) = \omega$ extending $(F,A^*,B^*)$ such that 
$x \prec_{\tilde{M}} y$ implies that $x < y$, $\prec_e^{\tilde{M}}$ is a stable-ish partial order and $\mathcal{K}^{\tilde{M}, A^*(\tilde{M}), B^*(\tilde{M})}$ 
is a requirement, we have that $\mathcal{K}^{\tilde{M}}$ is not essential below $p$.  

\begin{lem}
Fix a pair of indices $e$ and $i$ and let $\mathcal{K}^{G,A^*(G),B^*(G)}$ be the potential requirement specified by these indices.  
For any $p$, there is a dense set of $(F,A^*,B^*)$ such that either:
\begin{itemize}
\item $(F,A^*,B^*)\Vdash p\not\in\mathbb{P}^G_e$, or
\item $(F,A^*,B^*)\Vdash\mathcal{K}^{G}$ is not essential below $p$, or
\item $(F,A^*,B^*)\Vdash$ there is a split pair below $p$ in $\mathcal{K}^{G,A^*(G),B^*(G)}$.
\end{itemize}
\end{lem}

\begin{proof}
Fix $(F,A^*,B^*)$ and $p=(\sigma,\tau)$.  If there is any $(F',A',B')\leq(F,A^*,B^*)$ forcing that $p\not\in\mathbb{P}^G_e$ then we are done. 
So assume not, and assume that $(F,A^*,B^*) \Vdash p \in \mathbb{P}_e^G$.  

Suppose there is an extension $(F',A^*,B^*)\leq (F,A^*,B^*)$, sets $B_0>A_0>A^*\cup B^*$ and a split pair $q_0,q_1$ below $p$ such that 
$(F',A^*,B^*)\Vdash q_0,q_1\in\mathcal{K}^{F',A_0,B_0}$.  Let $A$ be the downwards closure of $A_0$ in $F'$ and $B$ the upwards closure of $B_0$ in $F'$.

We claim that $A$ is disjoint from $B \cup B^*$.  Fix $x \in A$ and $a \in A_0$ such that $x \preceq_{F'} a$.  First, suppose for a contradiction 
that $x \in B^*$ and hence $x \in F$.  If $a \in F$, then $x \preceq_F a$ and hence $a \in B^*$ because $B^*$ is closed upwards in $F$.  But, 
$a \in A_0$ and $A_0 > B^*$ giving a contradiction.  If $a \not \in F$, then $a \in F' \setminus F$, so $x \not \preceq_{F'} a$ since $x \in B^*$ and 
$(F',A^*,B^*) \leq (F,A^*,B^*)$, again giving a contradiction.  Therefore, $x \not \in B^*$.  Second, suppose for a contradiction that $x \in B$.  Then 
$y \preceq_{F'} x$ for some $y \in B_0$ and hence $y \preceq_{F'} a$.  Therefore, $y \leq a$ which contradicts $B_0 > A_0$.  Therefore, $A$ is disjoint 
from $B \cup B^*$.  

We also claim that $A^*$ is disjoint from $B \cup B^*$.  Fix $x \in A^*$ and note that $x \not \in B^*$ since $(F,A^*,B^*)$ is a condition and hence 
$A^* \cap B^* = \emptyset$.   Suppose for a contradiction that $x \in B$.  There is a $y \in B_0$ such that 
$y \preceq_{F'} x$ and hence $y \leq x$, which contradicts $B_0 > A^*$.  Therefore, $A^*$ is disjoint from $B \cup B^*$.  

Taken together, our claims show that $A \cup A^*$ is disjoint from $B \cup B^*$.  Since $A \cup A^*$ is downwards closed and 
$B \cup B^*$ is upwards closed, $(F',A^*\cup A,B^*\cup B)\leq (F,A^*,B^*)$ is a condition forcing the existence of a split pair below $p$ in 
$\mathcal{K}^{G,A^*(G),B^*(G)}$.

If there is no such $(F',A^*,B^*)\leq (F,A^*,B^*)$, we claim $(F,A^*,B^*)$ already forces that $\mathcal{K}^{G}$ is not essential below $p$: 
let $\tilde M$ be any completion of $F$ to a stable partial ordering satisfying the appropriate conditions from above, and suppose 
$\mathcal{K}^{\tilde{M}}$ were essential below $p$.  
Then in particular, there would be some $A_0>\max (A^*\cup B^*)$, some $B_0>\max A_0$, and a split pair $q_0,q_1$ over $p$ such that 
$q_0,q_1\in\mathcal{K}^{\tilde M,A_0,B_0}$.  But then there would have been some finite restriction $F'=\tilde M\upharpoonright[0,m]$ witnessing 
this, contradicting our assumption.
\end{proof}

Having verified that any generic for the ground forcing satisfies (C1), (C2) and (C3), we can give the proof of Theorem \ref{thm:ADS}.  

\begin{proof}
We iteratively build a Turing ideal $\mathcal{I}$ containing a partial order $M$, containing a solution to every infinite linear order in $\mathcal{I}$, 
but not containing any solution to $M$.

Let $M$ be a partial ordering generic for the ground forcing.  $M$ is stable by (C1), $M$ does not compute a solution to itself by (C2) and 
for each stable-ish linear order $\prec_e^M$, each requirement $\mathcal{K}^{M,A^*(M),B^*(M)} \subseteq \mathbb{P}_e^M$ is 
uniformly dense by (C3).  Thus, we have established the initial conditions for the iterated forcing with $X = M$.  For a fixed index $e$ 
such that $\prec_e^M$ is a stable-ish linear order, let $G$ be a generic solution to $\prec_e^M$ obtained from the iteration forcing.  
By Lemmas \ref{generic_iteration}, \ref{lem:nosolution} and \ref{lem:uniformdense}, $M \oplus G$ does not compute a solution to 
$M$ and for every stable-ish linear order $\prec_{e'}^{M \oplus G}$, each requirement $\mathcal{K}^{M \oplus G,A^*(M),B^*(M)} 
\subseteq \mathbb{P}_{e'}^{M \oplus G}$ is uniformly dense.  

Iterating this process (and choosing stable-ish partial orders systematically to ensure that we eventually consider each one) gives an ideal 
$\mathcal{I}$ with the property that whenever $\prec$ is a linear order in $\mathcal{I}$, either $\prec$ is stable-ish, and 
therefore we added a solution to $\mathcal{I}$ at some stage, or $\prec$ is not stable-ish, and so a solution is computable from $\prec$, 
and therefore belongs to $\mathcal{I}$.  We have ensured that $M\in\mathcal{I}$ but that no solution to $M$ belongs to $\mathcal{I}$.

Therefore $(\omega,\mathcal{I})$ is a model of $\RCA + \ADS$, but is not a model of $\SCAC$.
\end{proof}

\section{$\EM$ background}
\label{sec:Erdos} 

In this section, we present proofs of Theorems \ref{thm:KLSW} and \ref{thm:DKLS}, which are restated below for convenience.  
We begin with some basic properties of infinite transitive tournaments and their transitive subsets.  We regard 
every tournament $T$ (including finite subtournaments) as containing elements $-\infty$ and $\infty$ with the property that 
$T(-\infty,x)$ and $T(x, \infty)$ hold for every $x \in T$.  If $T$ is a transitive tournament, then the $T$ relation defines a linear order on the domain of 
$T$ with $-\infty$ as the least element and $\infty$ as the greatest element.  We will denote this order by $\leq_T$, or by $\leq_F$ if $F$ is a finite transitive 
subset of some ambient (nontransitive) tournament $T$.

\begin{defn}
Let $T$ be an infinite tournament and let $a,b \in T$ be such that $T(a,b)$ holds.   The \emph{interval} $(a,b)$ is the set of all $x \in T$ such that 
both $T(a,x)$ and $T(x,b)$ hold.  That is, $(a,b)$ is the set of points ``between" $a$ and $b$ in $T$.
\end{defn}

\begin{defn}
\label{def:minimal}
Let $F \subseteq T$ be a finite transitive subset of an infinite tournament $T$.  For $a,b \in F$ such that $T(a,b)$ holds (i.e.~$a \leq_F b$), we say 
$(a,b)$ is a \emph{minimal interval of} $F$ if there is no $c \in F$ such that $T(a,c)$ and $T(c,b)$ both hold (i.e.~$b$ is the successor of $a$ in $\leq_F$).
\end{defn}

In the context of Definition \ref{def:minimal}, $(a,b)$ is an interval in $T$ well as in $F$.  
However, the fact that $(a,b)$ is a minimal interval of $F$ is a property of this interval in $F$.

\begin{defn}
Let $T$ be an infinite tournament and $F \subseteq T$ be a finite transitive set.  $F$ is \emph{extendable} if $F$ is a subset of some solution to $T$.  
A \emph{one point transitive extension of} $F$ is a transitive set $F \cup \{ a \}$ such that $a \not \in F$.  
\end{defn}

\begin{lem}
\label{lem:onepoint}
Let $T$ be an infinite transitive tournament and $F \subseteq T$ be a finite transitive set.  $F$ is extendable if and only if $F$ has infinitely many 
one point transitive extensions.
\end{lem}

\begin{proof}
If $F$ is extendable, then it clearly has infinitely many one point extensions.  Conversely, suppose $F$ has infinitely many one point extensions.  
Let $T'$ be the set of all $a \in T \setminus F$ such that $F \cup \{ a \}$ is transitive.    Since $F$ is transitive,  we can list $F$ in $\leq_F$ order as 
\[
-\infty <_F x_0 <_F x_1 <_F \cdots <_F x_k <_F \infty
\]
Because $T'$ is infinite and there are finitely many minimal intervals in $F$, there must be a minimal interval $(a,b)$ of $F$ such that $(a,b) \cap T'$ is 
infinite.  (Note that $a$ could be $-\infty$, if there are infinitely many elements $a \in T'$ such that $T(a,x_0)$ holds.  Similarly, $b$ could be 
$\infty$.)  Fix such a minimal interval $(a,b)$ in $F$ and let $T'' = T' \cap (a,b)$.  $T''$ is an infinite subtournament of $T$ and hence (viewing $T''$ 
as an infinite tournament), $T''$ contains an infinite transitive tournament $T'''$.  Since $T'''$ is contained in a minimal interval of $F$, $T''' \cup F$ is 
transitive, and hence is a solution to $T$ containing $F$.  
\end{proof}

\begin{lem}
\label{lem:singleton}
Let $T$ be an infinite tournament.  For any $a \in T$, there is a solution to $T$ containing $a$.
\end{lem}

\begin{proof}
Fix $a \in T$ and let $F = \{ a \}$.  For all $b \in T$, $\{ a,b \}$ is a transitive, so $F$ has infinitely many one point transitive extensions.  By 
Lemma \ref{lem:onepoint}, $F$ is extendable.  
\end{proof}

\begin{lem}
Let $T$ be an infinite transitive tournament and let $F \subseteq T$ be a finite transitive extendible set.  Cofinitely many of the one 
point transitive extensions of $F$ are extendable. 
\end{lem}

\begin{proof}
Suppose for a contradiction that there are infinitely many $x \in T \setminus F$ such that $F \cup \{ x \}$ is transitive but not extendable.  Let 
$T'$ be the set of all such $x$.  As in the proof of Lemma \ref{lem:onepoint}, there must be a minimal interval $(a,b)$ of $F$ such that 
$T' \cap (a,b)$ is infinite.  Fix such an interval $(a,b)$ and let $T'' = T' \cap (a,b)$.  $T''$ is an infinite subtournament of $T$, so there is an infinite 
transitive set $T''' \subseteq T''$.  $F \cup T'''$ is a solution to $T$ containing $F$ as well as infinitely many point from $T'$ giving the desired contradiction.    
\end{proof}

\begin{thm}[Kach, Lerman, Solomon and Weber]
There is a computable infinite tournament $T$ with no infinite $\Sigma^0_2$ transitive subtournaments.
\end{thm}

\begin{proof}
Since every infinite $\Sigma^0_2$ set contains an infinite $\Delta^0_2$ subset, it suffices to construct an infinite computable tournament 
$T$ with no infinite $\Delta^0_2$ transitive subtournaments.  We build $T$ in stages to meet the 
following requirements. 
\[
R_e: \text{If } D_e(x) = \lim_s \varphi_e(x,s) \text{ exists for every } x, \text{ then } D_e \text{ is finite or } D_e \text{ is not transitive.}
\]  

As stage $s$, we determine whether $T(x,s)$ or $T(s,x)$ holds for each $x < s$ by acting in substages $e < s$.   At 
substage $e$, $R_e$ chooses the $2e+2$ least elements $x_0 < x_1 < \cdots < x_{2e+1}$ (less 
than $s$) that $\varphi_e$ currently claims are in $D_e$.  (If there are not $2e+2$ many such elements, then we proceed to the 
next substage.)  Let $x_i$ and $x_j$ be the least from this set which have not been chosen as witnesses by a higher priority requirement 
at this stage and assume that $T(x_i,x_j)$ holds.  Declare that $T(s,x_i)$ and $T(x_j,s)$ hold so that $\{ x_i, x_j, s \}$ is not transitive.  
Proceed to the next substage.  When all substages are complete, declare $T(x,s)$ for any $x < s$ for which we have not declared either 
$T(x,s)$ or $T(s,x)$.  This ends stage $s$.  

It is clear that this process defines a computable infinite tournament $T$.  To see that $R_e$ is met, assume that $D_e(x)$ is defined for all 
$x$.  Let $x_0 < x_1 < \cdots < x_{2e+1}$ be least such that $D_e(x_i) = 1$ and let $s$ be such that $\varphi_e$ claims that each 
$x_i$ is in $D_e$ for all $t \geq s$.  For every $t \geq s$, $R_e$ chooses a pair of element from $\{ x_0, \ldots, x_{2e+1} \}$ to make 
a cycle with $t$.  Therefore, $\{ x_0, \ldots, x_{2e+1} \}$ has only finitely many one point transitive extensions and hence is not a subset of 
any infinite transitive subtournament.  
\end{proof}

\begin{thm}[Dzhafarov, Kach, Lerman and Solomon]
There is a computable infinite tournament $T$ with no infinite hyperimmune-free transitive subtournaments.  
\end{thm}

\begin{proof}
  We build $T$ in stages to meet, for each~$e$, the requirement $R_e$ that if $\{ D_{\varphi_e(x)} \mid x \in \mathbb{N} \}$ 
  is a strong array, then there are $x_0 < x_1$ such that for all $y_0 \in D_{\varphi_e(x_0)}$ and all $y_1 \in D_{\varphi_e(x_1)}$, 
  the set $\{ y_0, y_1 \}$ is not extendible.  
 
  The strategy to meet an individual requirement $R_e$ in isolation is
  straightforward.  We wait for $\varphi_e(x_0)$ to converge for some $x_0$, and start
  defining $T(y,s)$ for all $y \in D_{\varphi_e(x_0)}$ and all $s$.  If $\{D_{\varphi_e} : e \in \omega\}$
  is a strong array, we must eventually find an $x_1$ such that $\varphi_e(x_1)$
  converges with $T(y_0,y_1)$ for all $y_0 \in D_{\varphi_e(x_0)}$ and all
  $y_1 \in D_{\varphi_e(x_1)}$.  We then start defining $T(s,y)$ for all $y \in D_{\varphi_e(x_0)}$
  and all $s$, and $T(y,s)$ for all $y \in D_{\varphi_e(x_1)}$ and all $s$.  Thus ensures that
  $R_e$ is met.  Sorting out competing requirements can be handled via a standard finite
  injury priority argument, as we now show.
  
  At stage $s$, we define $T(x,s)$ or $T(s,x)$ for all $x < s$.  We proceed
  by substages $e \leq s$.  At substage $e$, we act as follows, breaking into three cases.
  
  \medskip
  \noindent \emph{Case 1: $R_e$ has no witnesses.} Let $x_0$ be the least $x < s$, if it
  exists, such that
  \begin{enumerate}
    \item[(1)] $\varphi_{e,s}(x) \downarrow$;
    \item[(2)] $D_{\varphi_e(x)} \neq \emptyset$ and each element of $D_{\varphi_e(x)}$ is $< s$;
    \item[(3)] for all $i < e$ and any witness $y$ of $R_i$, $x > y$ and $D_{\varphi_e(x)}$ is disjoint
    from $D_{\varphi_i(y)}$.
  \end{enumerate}
  If there is no such $x_0$, then do nothing and proceed to the next substage.  If there is such an $x_0$, then 
  call $x_0$ the \emph{first witness} of $R_e$, define $T(y,s)$ for all $y \in D_{\varphi_e(x_0)}$, 
  cancel the witnesses of each $R_i$ with $i > e$ and proceed to the next substage.
 
  \medskip
  \noindent \emph{Case 2: $R_e$ has exactly one witness.}  Call this first witness $x_0$.  Let $x_1$
  be the least $x < s$, if it exists, that satisfies conditions (1)--(3) above, as well as
  \begin{enumerate}
  \setcounter{enumi}{3}
  \item[(4)] $T(y_0,y_1)$ for all $y_0 \in D_{\varphi_e(x_0)}$ and all $y_1 \in D_{\varphi_e(x)}$.
  \end{enumerate}
  If there is no such $x_1$, the define $T(y,s)$ for all $y \in D_{\varphi_e(x_0)}$ and proceed to the next substage.  If there is 
  such a witness $x_1$, then call $x_1$ the \emph{second witness} of $R_e$, define $T(s,y)$
  for all $y \in D_{\varphi_e(x_0)}$ and $T(y,s)$ for all $y \in D_{\varphi_e(x_1)}$, 
  cancel the witnesses of each $R_i$ with $i > e$ and proceed to the next substage.
  
  \medskip
  \noindent \emph{Case 3: $R_e$ has two witnesses.}  Let $x_0$ be the first witness and $x_1$ be the second witness.  Define 
  $T(s,y)$ for all $y \in D_{\varphi_e(x_0)}$ and $T(y,s)$ for all $y \in D_{\varphi_e(x_1)}$.  Proceed to the next substage.  
  
  \medskip
  \noindent When all substages $e < s$ are complete, define $T(x,s)$ for any $x < s$ for which neither $T(x,s)$ nor $T(s,x)$ has been defined.  
  This completes the description for the construction.  
  
  It is clear that $T$ is a computable tournament on $\mathbb{N}$.  
  To verify that requirement $R_e$ is met, suppose $\{D_{\varphi_e(x)} : x \in \mathbb{N} \}$
  is a strong array.  By induction, support further that each $R_i$, $i < e$, is satisfied.  Since each
  requirement $R_i$ has at most two witnesses at any stage, and since it can lose these witnesses
  only for the sake of some $R_{i'}$, $i' < i$, being assigned a witness, we let $s'$ be the least
  stage such that no $R_i$, $i < e$, is assigned a witness at any stage $s \geq s'$. By minimality of $s'$,
  it must be that $R_e$ has no witnesses at stage $s'$.  Since $\{D_{\varphi_e(x)} : x \in \mathbb{N} \}$ is a strong
  array, we let $s_0 \geq s'$ be the least stage such that some $x < s_0$ satisfies conditions (1)--(3) of
  the construction.  Then the least such $x$ is assigned as a first witness $x_0$ of $R_e$, and this
  witness is never cancelled.
  
  If, at any later stage $s_1 > s_0$, we assign a second witness $x_1$ for $R_e$, then $R_e$ will
  be satisfied.  (Because $x_1$ will never be canceled, we have $T(y_0,y_1)$, $T(s,y_0)$ and $T(y_1,s)$ 
  for all $s > s_1$, all $y_0 \in D_{\varphi_e(x_0)}$ and all $y_1 \in D_{\varphi_e(x_1)}$.  Therefore, $\{ y_0,y_1 \}$ is not extendible.)
  So suppose we never find a second witness $x_1$.  Then by
  construction, we define $T(y,s)$ for all $s \geq s_0$ and all  $y \in D_{\varphi_e(x_0)}$.  But if $s$ is large
  enough that for some $x < s$, $\varphi_{e,s}(x) \downarrow$ and all elements of $D_{\varphi_e(x)}$
  lie between $s_0$ and $s$, then $x$ will satisfy conditions (1)--(4) of the construction.  The
  least such $x$ is assigned as a second witness $x_1$ of $R_e$ for the desired contradiction.
\end{proof}

\section{$\EM$ does not imply $\SRT$}
\label{sec:EM}

Before giving the proof of Theorem \ref{thm:EM} in a style similar to the proof of Theorem \ref{thm:ADS}, we present some motivating ideas for the 
forcing construction. 
Fix an index $e$.  We sketch a strategy to meet a single diagonalization requirement towards constructing a stable coloring $c_{\infty}$ such that if 
$\Phi_e^{c_{\infty}}$ is the characteristic function for an infinite 
tournament $T_e^{c_{\infty}}$ given by $\Phi_e^{c_{\infty}}$, then there is a solution $S_{\infty}$ to $T_e^{c_{\infty}}$ such that 
$c_{\infty} \oplus S_{\infty}$ does not compute a solution to $c_{\infty}$.  A single diagonalization requirements has the form 
\[
R_i: \Phi_i^{c_{\infty} \oplus S_{\infty}} \text{ is not a solution to } c_{\infty}.
\]

To approximate $c_{\infty}$ we use a triples $(c,A^*,B^*)$ (called partial stable colorings) 
such that $c$ is a 2-coloring of the two element subsets of a finite domain $[0,|c|]$, and 
$A^*$ and $B^*$ are disjoint subsets of this domain.  We say $(c_{\alpha},A_{\alpha}^*,B_{\alpha}^*)$ \textit{extends} $(c_{\beta},A_{\beta}^*,B_{\beta}^*)$ if 
\begin{itemize}
\item $c_{\beta} \subseteq c_{\alpha}$, $A_{\beta}^* \subseteq A_{\alpha}^*$, $B_{\beta}^* \subseteq B_{\alpha}^*$, 
\item if $a \in A_{\beta}^*$ and $|c_{\beta}| < x \leq |c_{\alpha}|$, then $c_{\alpha}(a,x) = 0$, and 
\item if $b \in B_{\beta}^*$ and $|c_{\beta}| < x \leq |c_{\alpha}|$, then $c_{\alpha}(b,x) = 1$.
\end{itemize}
In the full construction, these partial stable colorings will be our ground forcing conditions, and we can force statements such as ``$F$ is a finite transitive 
subtournament of $T_e^{c_{\infty}}$'' or ``$I$ is a minimal interval in $F$ which is infinite in $T_e^{c_{\infty}}$" in a standard manner.  For example, 
the set of $(c,A^*,B^*)$ such that $i \in A^* \cup B^*$ is obviously dense, so a generic coloring $c_{\infty}$ will be stable.   
Given $\alpha = (c_{\alpha},A_{\alpha}^*,B_{\alpha}^*)$, we let $\mathbb{C}_{\alpha}$ denote the set of suitably generic infinite stable 
colorings extending $\alpha$. 

To approximate a solution $S_{\infty}$ to $T_e^{c_{\infty}}$, we augment a partial stable coloring $\alpha$ by adding a finite transitive subtournament 
$F_{\alpha}$ of $T_e^{c_{\alpha}}$ and a minimal interval $I_{\alpha}$ of $F_{\alpha}$ such that 
$I_{\alpha}$ is infinite in every tournament $T_e^{c_{\infty}}$ for $c_{\infty} \in \mathbb{C}_{\alpha}$.  $F_{\alpha}$ denotes the part of $S_{\infty}$ 
specified so far and $I_{\alpha}$ witnesses the fact that no matter how $c_{\alpha}$ is (generically) extended to $c_{\infty}$, 
$F_{\alpha}$ is extendible in $T_e^{c_{\infty}}$.  Thus, a condition for the purposes of this sketch has the form 
$\alpha = (c_{\alpha},A_{\alpha}^*,B_{\alpha}^*,F_{\alpha},I_{\alpha})$.  We say $\alpha$ extends $\beta$ if the partial colorings extend as above, 
$F_{\beta} \subseteq F_{\alpha}$, $I_{\alpha}$ is a subinterval of $I_{\beta}$ and for each $x \in F_{\alpha} \setminus F_{\beta}$, $x > \text{max}(F_{\beta})$ 
and $x \in I_{\beta}$. 

Given a condition $\alpha$, we would like to meet $R_i$ by extending $c_{\alpha}$ to $c_{\beta}$ and $F_{\alpha}$ to $F_{\beta}$ so that 
$\Phi_i^{c_{\beta} \oplus F_{\beta}}(y) = 1$ for some large $y \not \in A_{\alpha}^* \cup B_{\alpha}^*$.  Assuming we can do this without expanding 
$A_{\alpha}^* \cup B_{\alpha}^*$, we are free to add $y$ to either $A_{\alpha}^*$ or $B_{\alpha}^*$.  Therefore, if we can perform such an expansion 
twice, we will arrive at a condition $\gamma$ such that 
\[
\exists a \in A_{\gamma}^* \, \exists b \in B_{\gamma}^* \, (\Phi_i^{c_{\gamma} \oplus F_{\gamma}}(a) = \Phi_i^{c_{\gamma} \oplus F_{\gamma}}(b) = 1)
\]
and hence will have successfully diagonalized.  The obvious difficulty is that we have to maintain that $F_{\gamma}$ is extendible in $T_e^{c_{\infty}}$ 
for all $c_{\infty} \in \mathbb{C}_{\gamma}$.  We use following partition theorem to help address this problem.    

\begin{lem}
\label{lem:partition1}
Let $T$ be an infinite tournament, $F$ be a finite transitive set and $(a,b)$ be a minimal interval of $F$ which is infinite in $T$.  For any finite set $J \subseteq (a,b)$ 
such that $F \cup J$ is transitive, there is a partition $P \cup Q = J$ such that both $F \cup P$ and $F \cup Q$ are extendible and contain minimal intervals 
in $(a,b)$ which are infinite in $T$.  
\end{lem}

Given a condition $\alpha$, we ask our main question: is there a coloring $c_{\infty} \in \mathbb{C}_{\alpha}$ extending $c_{\alpha}$, 
an infinite transitive set $S$ in $T_e^{c_{\infty}}$ contained in $I_{\alpha}$ with $F_{\alpha} < S$, 
and a finite initial segment $J$ of $S$ such that for all partitions $P \cup Q = J$, there is a transitive 
$F \subseteq P$ or $F \subseteq Q$ for which $\Phi_i^{c_{\infty} \oplus (F_{\alpha} \cup F)}(y) = 1$ for some $y \not \in A_{\alpha}^* \cup B_{\alpha}^*$?

Suppose the answer to this question is yes.  We collect a finite set $Y$ disjoint from $A^*_\alpha\cup B^*_\alpha$ such that for each partition $P\cup Q=J$, there is some $F_{P,Q}\subseteq P$ or $F_{P,Q}\subseteq Q$ and some $y\in Y$ so that $\Phi^{c_\infty\oplus(F_\alpha\cup F_{P,Q})}_i(y)=1$.  Set $c_\beta$ to be an initial segment of $c_\infty$ extending $c_\alpha$ which is long enough to force all these computations, and set $A^*_\beta=A^*_\alpha\cup Y$ and $B^*_\beta=B^*_\alpha$ (or vice versa).

We have defined our extension $\beta$ of $\alpha$ with the exception of $F_\beta$ and $I_\beta$.  To define $F_\beta$, we need to find a partition $P\cup Q=J$ such that both $F_\alpha\cup P$ and $F_\alpha\cup Q$ are extendible in every $c_\infty\in\mathbb{C}_\beta$.  There are finitely many partitions and we consider each such partition in turn.  For a given partition $P\cup Q$, we ask whether there is an extension of $(c_\beta,A^*_\beta,B^*_\beta)$ which forces one of $F_\alpha\cup P$ or $F_\alpha\cup Q$ to be non-extendible.  (Forcing such a property depends only on the part of the condition dealing with the approximation of the coloring.)  If there is no such an extension then we take $F=F_{P,Q}$ and $F_\beta=F_\alpha\cup F_{P,Q}$.  If there is such an extension, replace $(c_\beta,A^*_\beta,B^*_\beta)$ by this extension and consider the next partition.  We must eventually find a partition which cannot be forced to be finite, since otherwise we end with a condition $(c_\beta,A^*_\beta,B^*_\beta)$ extending $(c_\alpha,A^*_\alpha,B^*_\alpha)$ which forces either $F_\alpha\cup P$ or $F_\alpha\cup Q$ to be non-extendible for any $P\cup Q=J$, and hence they must be non-extendible in any $c_\infty\in\mathbb{C}_\beta$, contradicting Lemma \ref{lem:partition1}.

To define $I_{\beta}$, we need to find a minimal subinterval of $I_{\alpha}$ which is infinite in every $c_{\infty} \in \mathbb{C}_{\beta}$.  $F$ subdivides $I_{\alpha}$ 
into finitely many minimal intervals and we follow the same procedure as before, considering each such minimal interval in turn.  For a given minimal interval of $I_{\alpha}$ (as subdivided 
by $F$), ask whether there is an 
extension of $(c_{\beta}, A_{\beta}^*, B_{\beta}^*)$ which forces this interval to be finite.  (Forcing such a property depends only on the part of the 
condition dealing with the approximation of the coloring.)  If there is no such extension, then set $I_{\beta}$ equal to this minimal interval and we have 
completed the description of $\beta$.  If there is such an extension, replace $(c_{\beta}, A_{\beta}^*, B_{\beta}^*)$ by this extension and consider the 
next minimal interval.  We must eventually find a minimal interval which cannot be forced to be finite because otherwise, we end with a condition 
$(c_{\beta},A_{\beta}^*,B_{\beta}^*)$ extending $(c_{\alpha},A_{\alpha}^*,B_{\alpha}^*)$ which forced all the minimal subintervals of $I_{\alpha}$ to be finite and 
hence forced $I_{\alpha}$ to be finite.  This contradicts the assumption that $I_{\alpha}$ was infinite in all $c_{\infty} \in \mathbb{C}_{\alpha}$.  

Thus, when the answer to the main question is yes, we are able to extend $\alpha$ to $\beta$ forcing $\Phi_i^{c_{\beta} \oplus F_{\beta}}(y) = 1$ 
for a large element $y$ which we could put into either $A_{\beta}^*$ or $B_{\beta}^*$.  So, suppose the answer to the main question is no.  To deal with this 
case, we introduce a Matthias component to our forcing conditions.  

Fix $c_{\infty} \in \mathbb{C}_{\alpha}$.  A solution $S$ to $T_e^{c_{\infty}}$ is consistent with the restrictions placed by $\alpha$ if $S$ is contained in 
$I_{\alpha}$ and $F_{\alpha} < S$.  We narrow this collection of solutions to consider only those that satisfy 
\[
\forall \, \text{finite} \, B \subseteq S \, \forall y \, (\Phi_i^{c_{\infty} \oplus (F_{\alpha} \cup B)}(y) = 1 \rightarrow y \in A_{\alpha}^* \cup B_{\alpha}^*).
\]
Let $\mathbb{S}_{\alpha,c_{\infty}}$ denote this collection of solutions to $T_e^{c_{\infty}}$.  (We show this set is nonempty below.)  
If our eventual generic coloring is 
$c_{\infty}$ and we restrict our choices for extensions of $F_{\alpha}$ to those which are subsets of elements of $\mathbb{S}_{\alpha,c_{\infty}}$, then our 
generic solution $S_{\infty}$ will meet $R_i$ because $\Phi^{c_{\infty} \oplus S_{\infty}}$ will be finite.  Therefore, we want to add a 
Matthias component to our conditions requiring that when we extend $c_{\alpha}$ and $F_{\alpha}$ to $c_{\beta}$ and $F_{\beta}$ in the future, we 
must have that $F_{\beta}$ is a subset of some element of $\mathbb{S}_{\beta,c_{\infty}}$ for each $c_{\infty}$ extending $c_{\beta}$.

\begin{lem}
$\mathbb{S}_{\alpha,c_{\infty}}$ is nonempty.
\end{lem}

\begin{proof}
Fix a solution $S = \{ s_1 < s_2 < \cdots \}$ of $T_e^{c_{\infty}}$ consistent with $\alpha$.  For $n \geq 1$, let $S_n = \{ s_1, \ldots, s_n \}$.  Let $R(P,Q,n)$ be a 
predicate that holds if and only if $P \cup Q = S_n$ is a partition such that for all $F$ contained in $P$ or $Q$ and all $y$, 
$\Phi_i^{c_{\infty} \oplus (F_{\alpha} \cup F)}(y) = 1$ implies $y \in A_{\alpha}^* \cup B_{\alpha}^*$.   Because the answer to the main question was no, 
for every $n$, there is a partition $P \cup Q = S_n$ such that $R(P,Q,n)$ holds.  We define a finitely branching tree $T$ such that the nodes at level $n$ of 
$T$ are labelled by the ``halves" of such partitions so that if $\sigma$ is an extension of $\tau$ on $T$, then the label of $\tau$ is a subset of the label of 
$\sigma$.  

For level $n=0$, let $T$ have a root $\lambda$ with label $\ell(\lambda) = \emptyset$.  For level $n=1$, there is only one partition $\{ s_1 \} \cup \emptyset$ of 
$S_1$ and therefore, this partition must satisfy $R$.  Let $T$ have two nodes $\sigma_0$ and $\sigma_1$ at level 1 with labels $\ell(\sigma_0) = \{ s_1 \}$ and 
$\ell(\sigma_1) = \emptyset$.    

For level $n+1$, consider each partition $P \cup Q = S_{n+1}$ satisfying $R(P,Q,n+1)$.  Let $P' \cup Q' = S_n$ be the restriction of this partition to $S_n$ and 
note that $R(P',Q',n)$ holds.  Therefore, by induction, there are nodes $\sigma'$ and $\tau'$ at level $n$ of $T$ such that $\ell(\sigma')=P'$ and 
$\ell(\tau') = Q'$.  Add nodes $\sigma$ and $\tau$ to $T$ at level $n+1$ with $\sigma$ a successor of $\sigma'$ and label $\ell(\sigma) = P$ and with 
$\tau$ a successor of $\tau'$ with label $\ell(\tau) = Q$.

This completes the description of $T$.  Notice that $T$ is infinite, finitely branching and has the property that if $\sigma$ is an extension of $\tau$, then 
$\ell(\tau) \subseteq \ell(\sigma)$.  To extract an element of $\mathbb{A}_{\alpha,c_{\infty}}$, we break into two cases.  

First, suppose there is a level $n$ and a node $\sigma$ at level $n$ such that for every $m \geq n$, there is an extension $\tau_m$ of $\sigma$ at level 
$m$ with $\ell(\tau_m) = \ell(\sigma)$.  In this case, for every $m \geq n$, $\ell(\sigma) \cup (S_m \setminus \ell(\sigma)) = S_m$ satisfies $R$.  
Therefore, $S' = \cup_{m \geq n} (S_m \setminus \ell(\sigma))$ is an infinite subset of $S$ and hence is a solution to $T_e^{c_{\infty}}$ consistent 
with $\alpha$.  Furthermore, if $B \subseteq S'$ is finite, then $B \subseteq (S_m \setminus \ell(\sigma))$ for some $m$ and hence if 
$\Phi_i^{c_{\infty} \oplus (F_{\alpha} \cup B)}(y) = 1$, then $y \in A_{\alpha}^* \cup B_{\alpha}^*$.  Therefore, $S' \in \mathbb{S}_{\alpha,c_{\infty}}$.

Second, suppose there is no such $\sigma$.  In this case, for every $\sigma$ in $T$, either the tree above $\sigma$ is finite or there is a level 
$m$ such that for all $\tau$ extending $\sigma$ at level $m$, $\ell(\sigma) \subsetneq \ell(\tau)$.  Because $T$ is infinite and finitely branching, it has an 
infinite path $\sigma_0 \subseteq \sigma_1 \subseteq \cdots$.  Let $S' = \cup \ell(\sigma_k)$.  By our case assumption, $S'$ is an infinite subset of 
$S$ and as above, $S' \in \mathbb{S}_{\alpha,c_{\infty}}$.
\end{proof}

To fully implement the strategy sketched in this section, we need a more uniform method to represent these collections of ``allowable'' solutions to 
$T_e^{c_{\infty}}$.  We wait for the full construction to spell out the details of representing the solutions in an 
appropriate matter. 

\subsection{Iteration forcing}\label{sec:it2}

As with the proof of Theorem \ref{thm:ADS}, we begin the formal proof of Theorem \ref{thm:EM} with the iteration forcing.  Assume we have already 
used the ground forcing to build a stable 2-coloring $c$ of $[\mathbb{N}]^2$ along with $A^*(c)$ and $B^*(c)$.  The general context for one 
step of the iteration forcing is a fixed set $X$ and an index $e$ such that 
\begin{itemize}
\item $c \leq_T X$;
\item $X$ does not compute a solution to $c$;
\item $\Phi_e^X$ is the characteristic function for an infinite tournament $T_e^X$ on $\mathbb{N}$; and 
\item each requirement $\mathcal{K}^{X,A^*(c),B^*(c)}$ is uniformly dense (in a sense defined below).
\end{itemize}

Our goal is to define a generic solution $G$ to $T_e^X$ such that $X \oplus G$ does not compute a solution to $c$ and the requirements 
$\mathcal{K}^{X \oplus G,A^*(c),B^*(c)}$ are uniformly dense.  We let $\mathbb{E}_e^X$ denote the set of finite subtournaments of 
$T_e^X$ and let variations on $E$ (such as $E'$, $\tilde{E}$, $E_0$, etc.) denote finite subtournaments of $T_e^X$.  We let 
$\mathbb{F}_e^X$ denote the set of finite transitive subtournament of $T_e^X$ and let variations on $F$ denote finite transitive subtournaments of $T_e^X$.  

\begin{defn}
\label{def:family}
  A \emph{family of subtournaments} of $T_e^X$ (or simply a \emph{family}) is a function $S^X(n)$ computable from $X$ such that
  \begin{enumerate}
  \item[(A1)] for each $n$, $S^X(n)$ is a finite collection of finite subtournaments of $T$; and
  \item[(A2)] whenever $E \in S^X(n+1)$, there is an $E' \in S^X(n)$ such that $E' \subsetneq E$ and for all $x \in E \setminus E'$, 
  $x > \text{max}(S^X(n)) = \text{max}( \cup \{ \tilde{E} \mid \tilde{E} \in S^X(n) \})$.   
  \end{enumerate}
We say $S$ is \emph{infinite} if for every $n$, $S(n)$ is non-empty.
\end{defn}

To be more precise, the value of $S^X(n)$ is the canonical index for the finite collection of finite subtournaments so we can calculate 
$\text{max}(S^X(n))$ effectively in $X$. A family of subtournaments is represented by its index, so later when such families appear as 
part of our forcing conditions, we mean that the index is listed a part of the condition.  Note that there is a trivial infinite family of  
subtournaments given by $S(n)=\{[0,n]\}$.

For the first part of this construction, we will typically drop the 
superscript $X$ and denote a family by $S$.  Later when we consider forcing uniform density of the requirements 
$\mathcal{K}^{X \oplus G,A^*(c),B^*(c)}$ at the next 
level, we will be more careful about the oracle as it will have the form $X \oplus F$ for some finite transitive subtournament $F$ of $T_e^X$.  

We view a family of subtournaments $S$ as a finitely 
branching tree whose nodes at level $n$ are the finite subtournaments in $S(n)$.  (We artificially add a root node at the bottom of the tree.)  
The node $E \in S(n+1)$ is a successor of $E' \in S(n)$ if and only if $E' \subseteq E$ and for all $x \in E \setminus E'$, $x > \text{max}(S(n))$.  
Since there is a unique such node $E'$, the family $S$ forms a finitely branching tree under this successor relation, and $X$ can compute the 
number of branches at each level.

If $E_0 \subsetneq E_1 \subsetneq E_2 \subsetneq \cdots$ is an infinite path in a family $S$ (viewed as a tree), then $\cup E_n$ is an infinite subtournament of 
$T_e^X$.  We say that an infinite subtournament $U \subseteq T_e^X$ is \emph{coded by} $S$ if $U = \cup E_n$ for some infinite path through $S$.  
We will use families in our forcing conditions in the style of Matthias forcing to restrict our attention to solutions of $T_e^X$ which are contained 
in an infinite subtournament coded by $S$.     

\begin{defn}
Let $E$ be a finite subtournament of $T_e^X$ and let $S$ be a family of subtournaments such that $E < E'$ for all $E' \in S(0)$.  
We write $E+S$ for the family given by 
\[
(E+S)(n)=\{E \cup E' \mid E' \in S(n) \}.
\]  
\end{defn}

It is straightforward to check that under the conditions of this definition, $E+S$ is a family of subtournaments and that 
the subtournaments coded by $E+S$ are exactly the subtournaments coded by $S$ unioned with $E$.

\begin{defn}
Let $S'$ and $S$ be families of subtournaments.  
We say $S'\leq S$ if for every $n$, there is an $m$ such that whenever $E' \in S'(n)$, there is an $E \in S(m)$ with $E' \subseteq E$.
\end{defn}

If $S' \leq S$, then $S'$ is similar to a subtree of $S$, but is more flexible.  For example, let $S(n) = \{[0,n]\}$ and let $S'(n) = \{ \text{Even}(n), 
\text{Odd}(n) \}$ where $\text{Even}(n)$ is the set of even numbers less than $n$ and $\text{Odd}(n)$ is the set of odd numbers less than $n$.  Then 
$S' \leq S$ despite the fact that $S$ contains a single branch and $S'$ contains two branches.  

\begin{lem}
Let $S$ and $S'$ be infinite families of subtournaments such that $S' \leq S$.  Every infinite subtournament coded by $S'$ is contained in an infinite 
subtournament coded by $S$.  
\end{lem}

\begin{proof}
Let $E_0' \subsetneq E_1' \subsetneq \cdots$ be an infinite path in $S'$.  Let $m(k)$ be a function such that 
for each $k$, $E_k' \subseteq E$ for some $E \in S(m(k))$.  Consider the subtree of $S$ formed by taking the downward closure of all 
$E \in S(m(k))$ such that $E_k' \subseteq E$.  This subtree is infinite because $S' \leq S$ and therefore has an infinite path 
$E_0 \subsetneq E_1 \subsetneq \cdots$ because it is finitely branching.  By the definition of the subtree, $\cup E_k' \subseteq \cup E_k$ and 
hence the subtournament coded by $S'$ is contained in a subtournament coded by $S$.  
\end{proof}  

\begin{defn}
Our set of forcing conditions, $\mathbb{Q}^X_e$, is the set of triples $(F,I,S)$ such that
  \begin{itemize}
  \item $F$ is a finite, transitive subtournament of $T_e^X$,
  \item $I$ is a minimal interval of $F$, and
  \item $S$ is an infinite family of subtournaments such that $F < E$ for all $E \in S(0)$ and such that for each $n$, $E \subseteq I$ 
  for all $E \in S(n)$.
  \end{itemize}
We say $(F',I',S')\leq (F,I,S)$ if
\begin{itemize}
\item $F\subseteq F'$,
\item $I'$ is a sub-interval of $I$, and
\item $(F'\setminus F)+S'\leq S$.
\end{itemize}
\end{defn}

The last condition implies that there is an $n$ and an $E\in S(n)$ such that $F'\setminus F\subseteq E$.

There are several points worth noting about this definition.  First, for each $(F,I,S) \in \mathbb{Q}_e^X$, $F$ is extendible in $T_e^X$.  In particular, if 
$U$ is a subtournament coded by $S$ and $S \subseteq U$ is an infinite transitive set, then $F \cup S$ is a solution of $T_e^X$ because $S$ is contained 
in the minimal interval $I$ of $F$.  

Second, if $(F',I',S') \leq (F,I,S)$, then $F < x$ for all $x \in F' \setminus F$.  To see why, by the third condition in the definition of an extension, we can 
fix $n$ such that $F' \setminus F \subseteq \widehat{E}$ for some $\widehat{E} \in S(n)$.  Since $(F,I,S)$ is a condition, $F < E$ for all $E \in S(0)$.  
Applying Condition (A2) in Definition \ref{def:family}, $F < \widehat{E}$ and hence $F < (F' \setminus F)$, completing the explanation.  If $F, F' \in 
\mathbb{F}_e^X$, then we say $F'$ \textit{extends} $F$ if $F \subseteq F'$ and $F < (F' \setminus F)$.  Thus, if $(F',I',S') \leq (F,I,S)$, then $F'$ extends $F$.

Third, let $(F',I',S') \leq (F,I,S)$ and let $U$ be a subtournament coded by $S'$.  Because $(F' \setminus F) + S' \leq S$, there is a subtournament 
$V$ coded by $S$ such that $(F' \setminus F) \cup U \subseteq V$.  Therefore, when passing from $(F,I,S)$ to $(F',I',S')$, we extend the fixed 
part of our transitive solution from $F$ to $F'$ while narrowing the interval in $T_e^X$ where we can look for future transitive extensions from $I$ to $I'$ 
and further restricting where in $I'$ we can find these extensions in a way that is 
consistent with the restriction to the previous family $S$.

Our construction of the generic solution $G$ to $T_e^X$ will be spelled out in detail later, but it proceeds in the usual way.  We specify a sequence of conditions 
$(F_0,I_0,S_0) \geq (F_1, I_1, S_1) \geq \cdots$ meeting the appropriate requirements and set $G = \cup_n F_n$.  

Having defined our forcing conditions, we turn to the requirements and the notion of uniform density.  

\begin{defn}
A \emph{requirement} is a set $\mathcal{K}^{X,A^*(c),B^*(c)}$ together with a partial function $a_{\mathcal{K}}^X$ meeting the 
following conditions.
\begin{itemize}  
\item $\mathcal{K}^{X,A^*(c),B^*(c)}$ is a set of finite transitive subtournaments of $T_e^X$ which is closed under extensions and is defined by 
\[
\mathcal{K}^{X,A^*(c),B^*(c)} = \{ F \in \mathbb{F}_e^X \mid \exists a \in A^*(c) \, \exists b \in B^*(c) \, (R^X_{\mathcal{K}}(F,a,b)) \}
\]
for an $X$-computable relation $R^X_{\mathcal{K}}(x,y,z)$ which is symmetric in the $y$ and $z$ variables.  

\item $a_{\mathcal{K}}^X$ is a partial $X$-computable function with $\text{domain}(a_{\mathcal{K}}^X) \subseteq \mathbb{F}_e^X$ such that 
if $F'$ extends $F$ and $a_{\mathcal{K}}^X(F)$ converges, then $a_{\mathcal{K}}^X(F') = a_{\mathcal{K}}^X(F)$.  
\end{itemize}
\end{defn}

Note that $\mathcal{K}^{X,A^*(c),B^*(c)}$ is specified by a pair of indices, one for the relation $R^X_{\mathcal{K}}$ and one for the partial function 
$a_{\mathcal{K}}^X$.  

\begin{example}
\label{example:diag}
For each $m$, we define the requirement 
\[
\mathcal{W}_m^{X,A^*(c),B^*(c)} = \{ F \in \mathbb{F}_e^X \mid \exists a \in A^*(c) \, \exists b \in B^*(c) \, ( \Phi_{m}^{X \oplus F}(a) = 
\Phi_{m}^{X \oplus F}(b) = 1) \}
\]
together with the function $a_{\mathcal{W}_m}^X(F) = u$ where $u$ is the first number (if any) on which 
$\Phi_{m}^{X \oplus F}(u) = 1$ in appropriately dovetailed computations.  

Suppose a condition $(F,I,S)$ used to construct our generic $G$ satisfies $F \in \mathcal{W}_m^{X,A^*(c),B^*(c)}$.  Because $F$ is an initial segment of 
$G$, we have successfully diagonalized against $\Phi_m^{X \oplus G}$ computing a solution to $c$.  On the other hand, if $a_{\mathcal{W}_m}^X(F)$ 
diverges for every condition $(F,I,S)$ in the sequence defining $G$, then $\Phi_m^{X \oplus F}$ is empty for every initial segment $F$ of $G$ and 
hence $\Phi_m^{X \oplus G}$ does not compute a solution to $c$.  
\end{example}

As in the previous construction, we can replace the sets $A^*(c)$ and $B^*(c)$ in a requirement $\mathcal{K}^{X,A^*(c),B^*(c)}$ by 
finite sets $A$ and $B$ and consider the set
\[
\mathcal{K}^{X,A,B} = \{ F \in \mathbb{F}_e^X \mid \exists a \in A \, \exists b \in B \, (R^X_{\mathcal{K}}(F,a,b)) \}.
\]
Typically, we will work in this context with $A = \{ a_{\mathcal{K}}^X(F) \}$ for some fixed $F$ for which $a_{\mathcal{K}}^X(F)$ converges.  We 
abuse notation and write $\mathcal{K}^{X,a_{\mathcal{K}}^X(F),B}$ in this situation.

\begin{defn}
\label{def:EMessential}
We say $\mathcal{K}^{X}$ is \emph{essential} below $(F,I,S)$ if $a_{\mathcal{K}}^X(F)$ converges and for every $x$ there is a finite set 
$B>x$ and a level $n$ such that whenever $E \in S(n)$ and $E = E_0\cup E_1$ is a partition, there is an $i\in\{0,1\}$ and a transitive $F' \subseteq E_i$ such that $F\cup F' \in \mathcal{K}^{X,a_{\mathcal{K}}^X(F),B}$.
\end{defn}

\begin{defn}
\label{def:EMud}
We say $\mathcal{K}^{X,A^*(c),B^*(c)}$ is \emph{uniformly dense} if whenever $\mathcal{K}^{X}$ is essential below $(F,I,S)$, there is some 
level $n$ such that whenever $E \in S(n)$ and $E = E_0\cup E_1$ is a partition, there is an $i\in\{0,1\}$ and a transitive $F' \subseteq E_i$ such that 
$F\cup F' \in \mathcal{K}^{X,A^*(c),B^*(c)}$.
\end{defn}

\begin{defn}
\label{def:settle}
We say $(F,I,S)$ \emph{settles} $\mathcal{K}^{X,A^*(c),B^*(c)}$ if $a_{\mathcal{K}}^X(F)$ converges and 
either $F\in \mathcal{K}^{X,A^*(c),B^*(c)}$ or there is an $x$ such that 
whenever $E \in S(n)$ is on an infinite path through $S$ and $F' \subseteq E$ is transitive, $F\cup F' \not\in \mathcal{K}^{X,a_{\mathcal{K}}^X(F),(x,\infty)}$.
\end{defn}

We give one example to illustrate settling and prove one essential property of this notion.

\begin{example}
\label{example:diag2}
Suppose $(F,I,S)$ settles $\mathcal{W}_m^{X,A^*(c),B^*(c)}$.  We claim that if $(F,I,S)$ appears in a sequence defining a generic $G$,  then 
$\Phi_m^{X \oplus G}$ is not a solution to $c$.  If $(F,I,S) \in \mathcal{W}_m^{X,A^*(c),B^*(c)}$,  then this claim was verified in Example \ref{example:diag}.  
So, assume that $(F,I,S)$ settles $\mathcal{W}_m^{X,A^*(c),B^*(c)}$ via the second condition in this definition and fix the witness $x$.  We claim 
that for all $(\tilde{F},\tilde{I},\tilde{S}) \leq (F,I,S)$ and all $b > x$, $\Phi_m^{X \oplus \tilde{F}}(b) \neq 1$.  It follows immediately from this claim that 
$\Phi_m^{X \oplus G}$ is finite and hence is not a solution to $c$.  

To prove the claim, fix $(\tilde{F}, \tilde{I}, \tilde{S}) \leq (F,I,S)$.  By the definition of $a_{\mathcal{W}_m}^X(F)$, we know that 
$\Phi_m^{X \oplus F}(a_{\mathcal{W}_m}^X(F)) = 1$.  Since $\tilde{F}$ extends $F$, we have 
$\Phi_m^{X \oplus \tilde{F}}(a_{\mathcal{W}_m}^X(F)) = 1$.  Suppose for a contradiction that there is a 
$b > x$ such that $\Phi_m^{X \oplus \tilde{F}}(b) = 1$.  Then 
\[
\exists b > x \, \big(\Phi_m^{X \oplus \tilde{F}}(a_{\mathcal{W}_m}^X(F)) = \Phi_m^{X \oplus{F}}(b) = 1 \big)
\]
and hence $\tilde{F} \in \mathcal{W}_m^{X,a_{\mathcal{W}_m}^X(F),(x,\infty)}$.

Let $F' = \tilde{F} \setminus F$, so $F \cup F' \in \mathcal{W}_m^{X,a_{\mathcal{W}_m}^X(F),(x,\infty)}$.  Because $(\tilde{F}, \tilde{I}, \tilde{S}) \leq (F,I,S)$, 
we have $(\tilde{F} \setminus F) + S' = F' + S' \leq S$ and hence there is a level $n$ and an $E \subseteq S(n)$ such that $F' \subseteq E$.  Therefore, 
$F'$ shows that our fixed $x$ does not witness the second condition for $(F,I,S)$ to settle $\mathcal{W}_m^{X,A^*(c),B^*(c)}$ giving the desired 
contradiction.
\end{example}

\begin{lem}
\label{lem:downsettle}
If $(F,I,S)$ settles $\mathcal{K}^{X,A^*(c),B^*(c)}$ and $(\tilde{F},\tilde{I},\tilde{S}) \leq (F,I,S)$, then $\tilde{(F},\tilde{I},\tilde{S})$ settles $\mathcal{K}^{X,A^*(c),B^*(c)}$.
\end{lem}

\begin{proof}
Assume $(F,I,S)$ settles $\mathcal{K}^{X,A^*(c),B^*(c)}$ and fix $(\tilde{F},\tilde{I},\tilde{S}) \leq (F,I,S)$.  Since $a_{\mathcal{K}}^X(F)$ converges and 
$\tilde{F}$ is an extension of $F$, we know that $a_{\mathcal{K}}^X(\tilde{F})$ converges.

If $F \in \mathcal{K}^{X,A^*(c),B^*(c)}$, then since 
$\mathcal{K}^{X,A^*(c),B^*(c)}$ is closed under extensions, $\tilde{F} \in \mathcal{K}^{X,A^*(c),B^*(c)}$ and $(\tilde{F},\tilde{I},\tilde{S})$ settles 
$\mathcal{K}^{X,A^*(c),B^*(c)}$.

On the other hand, suppose $(F,I,S)$ settles $\mathcal{K}^{X,A^*(c),B^*(c)}$ via the second condition in Definition \ref{def:settle} with the witness $x$. 
We claim that the same witness $x$ works to show that $\tilde{(F},\tilde{I},\tilde{S})$ settles $\mathcal{K}^{X,A^*(c),B^*(c)}$.  Suppose for a contradiction 
that there is an $\tilde{E} \in \tilde{S}(\tilde{n})$ on an infinite path through $\tilde{S}$ and a transitive $\tilde{F}' \subseteq \tilde{E}$ such that 
$\tilde{F} \cup \tilde{F}' \in \mathcal{K}^{X,a_{\mathcal{K}}^X(\tilde{F}),(x,\infty)}$.  Note that $\tilde{F} \cap \tilde{F}' = \emptyset$ because 
$\tilde{F}' \subseteq \tilde{E} \in \tilde{S}(\tilde{n})$.  

Let $F' = (\tilde{F} \cup \tilde{F}') \setminus F$, so $F \cup F' \in \mathcal{K}^{X,a_{\mathcal{K}}^X(\tilde{F}),(x,\infty)} = 
\mathcal{K}^{X,a_{\mathcal{K}}^X(F),(x,\infty)}$.  Because $\tilde{F} \cap \tilde{F}' = \emptyset$, we have $F' = (\tilde{F} \setminus F) \cup \tilde{F}'$.  
It follows from $(\tilde{F} \setminus F) + \tilde{S} \leq S$ that there is a level $n$ and an $E \in S(n)$ on an infinite path in $S$ such that 
$(\tilde{F} \setminus F) \cup \tilde{E} \subseteq E$.  Therefore, $F' \subseteq E \in S(n)$ and $F \cup F' \in \mathcal{K}^{X,a_{\mathcal{K}}^X(F),(x,\infty)}$ 
contradicting the fact that $x$ was a witness for $(F,I,S)$ settling $\mathcal{K}^{X,A^*(c),B^*(c)}$ via the second condition in Definition \ref{def:settle}.
\end{proof}

The heart of this construction is the following theorem.

\begin{thm}
\label{thm:settle}
Let $\mathcal{K}^{X,A^*(c),B^*(c)}$ be a uniformly dense requirement and let $(F,I,S)$ be a condition for which $a_{\mathcal{K}}^X(F)$ converges.  
There is an extension $(F',I',S') \leq (F,I,S)$ settling $\mathcal{K}^{X,A^*(c),B^*(c)}$.  
\end{thm}

We prove Theorem \ref{thm:settle} at the end of this subsection after showing how it is used to construct our generic $G$ and verifying that $X \oplus G$ 
does not compute a solution to $c$ and that for any index $e'$ such that $\Phi_{e'}^{X \oplus G}$ defines an infinite tournament, 
the associated requirements $\mathcal{K}^{X \oplus G, A^*(c),B^*(c)}$ are uniformly dense.  

To define $G$, let $\mathcal{K}_n^{X,A^*(c),B^*(c)}$, for $n \in \omega$, be a list of all the requirements.  We define a sequence of conditions 
\[
(F_0,I_0,S_0) \geq (F_1,I_1,S_1) \geq \cdots
\]
by induction.  Let $F_0 = \emptyset$, $I_0 = (-\infty,\infty)$ and $S_0(n) = \{[0,n]\}$.  Assume $(F_k,I_k,S_k)$ has been defined.  Let $n$ be the least index 
such that $a_{\mathcal{K}_n}^X(F_k)$ converges and $\mathcal{K}_n^{X,A^*(c),B^*(c)}$ is not settled by $(F_k,I_k,S_k)$.  Applying Theorem \ref{thm:settle}, 
we choose $(F_{k+1},I_{k+1},S_{k+1})$ so that it settles $\mathcal{K}_n^{X,A^*(c),B^*(c)}$.  We define our generic by $G  = \cup F_n$.

The next lemma shows that we eventually settle each condition that is not trivially satisfied.  

\begin{lem}
\label{lem:settle}
Let $\mathcal{K}_n^{X,A^*(c),B^*(c)}$ be a requirement and let $(F_j,I_j,S_j)$ be the sequence of conditions defining $G$.  If there is an $i$ such 
that $a_{\mathcal{K}_n}^X(F_i)$ converges, then there is an index $k$ such that $(F_k,I_k,S_k)$ settles $\mathcal{K}_n^{X,A^*(c),B^*(c)}$.  
\end{lem}

\begin{proof}
Fix the least $i$ such that $a_{\mathcal{K}_n}^X(F_i)$ converges.  For all $j \geq i$, $a_{\mathcal{K}_n}^X(F_j)$ converges because $F_j$ extends $F_i$. 
By Lemma \ref{lem:downsettle}, once a requirement $\mathcal{K}_n^{X,A^*(c),B^*(c)}$ has been settled by a condition in the sequence 
defining $G$, it remains settled by all future conditions.  Therefore there must be a stage $k \geq i$ at which for all $n'<n$ such that $a_{\mathcal{K}_{n'}}^X(F_k)$ converges,  $\mathcal{K}_{n'}^{X,A^*(c),B^*(c)}$ is settled by $(F_k,I_k,S_k)$.  Either $\mathcal{K}_n^{X,A^*(c),B^*(c)}$ is already settled by $(F_k,I_k,S_k)$, or, by construction, $(F_{k+1}, I_{k+1}, S_{k+1})$ 
settles $\mathcal{K}_n^{X,A^*(c),B^*(c)}$.
\end{proof}

We can now verify the properties of $G$ starting with the fact that $X \oplus G$ does not compute a solution to $c$.  

\begin{lem}
\label{lem:nosolve}
$X \oplus G$ does not compute a solution to $c$.
\end{lem}

\begin{proof}
Fix an index $m$ and we show that $\Phi_m^{X \oplus G}$ is not a solution to $c$ using the requirement $\mathcal{W}_m^{X,A^*(c),B^*(c)}$.  
If $\Phi_m^{X \oplus G}(u)$ is never equal to $1$, then $\Phi_m^{X \oplus G}$ does not compute an infinite set and 
we are done.  Therefore, assume that $\Phi_m^{X \oplus G}(u)=1$ for some $u$.  In this case, $a_{\mathcal{W}_m}^X(F_i)$ converges for some $i$ and hence   
$\mathcal{W}_m^{X,A^*(c),B^*(c)}$ is settled by some condition $(F_k,I_k,S_k)$ in the sequence defining $G$.  In Example 
\ref{example:diag2} we verified that if $\mathcal{W}_m^{X,A^*(c),B^*(c)}$ is settled by a condition in a sequence defining a generic $G$, then 
$\Phi_m^{X \oplus G}$ does not compute a solution to $c$.
\end{proof}

Next, we describe the requirements forcing uniform density at the next level.  To specify a potential requirement at the next level, we need to fix three 
indices: an index $e'$ for a potential infinite transitive tournament $T_{e'}^{X \oplus G}$ and indices for the relation $R_{\mathcal{K}}^{X \oplus G}$ 
(defining $\mathcal{K}^{X \oplus G,A^*(c),B^*(c)}$) and the function $a_{\mathcal{K}}^{X \oplus G}$.  We regard the indices for 
$R_{\mathcal{K}}^{X \oplus G}$ and $a_{\mathcal{K}}^{X \oplus G}$ as a pair $\mathcal{K}$ and will represent this choice of indices by indicating $e'$ and 
$\mathcal{K}$.  For each choice of these indices and each $q = (F_q, I_q, S_q)$, representing a potential condition in $\mathbb{Q}_{e'}^{X \oplus G}$, we will 
have a requirement $\mathcal{T}_{e',\mathcal{K},q}^X$.  

To describe these requirements, we fix some forcing definitions for a fixed index $e'$.  
These definitions are exactly what one expects and in each case reflect that $F$ is a 
big enough to guarantee the convergence of relevant computations for any generic $G$ extending $F$. In these definitions,   
$F$ is a finite transitive subtournament of $T_e^X$.  $E$, $F'$ and similar variations are finite sets which are potentially subtournaments or transitive 
subtournaments of $T_{e'}^{X \oplus G}$.  For $q = (F_q, I_q, S_q)$, we assume $F_q$ is a finite set, $I_q$ is a pair of distinct elements of $F_q$ and $S_q$ is an 
index for a potential infinite family of subtournaments in $T_{e'}^{X \oplus G}$.
\begin{itemize}

\item $F \Vdash (E$ is a subtournament) if for every $u \in E$, $\Phi_{e'}^{X \oplus F}(u,u) = 0$ and for each distinct 
$u,v \in E$, there is an $i \in \{ 0, 1 \}$ such that $\Phi_{e'}^{X \oplus F}(u,v) = i$ and $\Phi_{e'}^{X \oplus F}(v,u) = 1-i$.

\item $F \Vdash (E$ is a not a subtournament) if either there is a $u \in E$ such that $\Phi_{e'}^{X \oplus F}(u,u) = 1$ or there are
distinct $u,v \in E$ such that $\Phi_{e'}^{X \oplus F}(u,v) =\Phi_{e'}^{X \oplus F}(v,u)$.

\item $F \Vdash (E$ is transitive) if $F \Vdash (E$ is a subtournament) 
and the relation on $E$ given by $\Phi_{e'}^{X \oplus F}$ is transitive.  

\item $F \Vdash (E$ is not transitive) if $F \Vdash (E$ is a subtournament) and the relation on $E$ given by $\Phi_{e'}^{X \oplus F}$ is not transitive.

\item Fix an index $S$ for a potential family of subtournaments of $T_{e'}^{X \oplus G}$. 
\begin{itemize}

\item $F \Vdash (S$ is a potential family at level $m$) if $S^{X \oplus F}(m)$ converges and outputs 
a finite set $\{ E_0, \ldots, E_k \}$ such that $F \Vdash (E_j$ is a subtournament) for each $j$.

\item $F \Vdash (S$ is a family up to level $n$) if for all $m \leq n$, $F \Vdash (S(m)$ is a potential family at 
level $m$), and if $m < n$, then the families of finite sets given by 
$S^{X \oplus F}(m)$ and $S^{X \oplus F}(m+1)$ satisfy Condition (A2) in Definition \ref{def:family}.  

\item $F \Vdash (S$ is not a family) if either there is an $m \leq |F|$ such that 
$S^{X \oplus F}(m) = \{ E_0, \ldots, E_k \}$ and for some $j$, $F \Vdash (E_j$ is not a subtournament), or there is an $m < |F|$ such that 
$F \Vdash (S$ is a family up to level $m$), $F \Vdash (S$ is a potential family at level $m+1$), but 
the families of finite sets given by $S^{X \oplus F}(m)$ and $S^{X \oplus F}(m+1)$ do not meet Condition (A2) in Definition \ref{def:family}.
\end{itemize}

\item $F \Vdash (q \not \in \mathbb{Q}_{e'}^{X \oplus G})$ if either $F \Vdash (F_q$ is not transitive); or $F \Vdash (F_q$ is transitive) but 
$I_q$ is not a minimal interval in $F_q$; or $F \Vdash (S_q$ is not a family); or there is an 
$E \in S_q^{X \oplus F}(0)$ such that $F_q \not < E$; or there is a $E \in S_q^{X \oplus F}(n)$ for some $n \leq |F|$ such that 
$E$ is not contained in $I_q$.

\item $F \Vdash (q$ is a condition to level $n$) if $F \Vdash (F_q$ is transitive; $I_q$ is a minimal interval in $F_q$; 
$F \Vdash (S$ is a family up to level $n$); for all $E \in S_q^{X \oplus F}(0)$, $F_q < E$; and for all $m \leq n$ and all 
$E \in S_q^{X \oplus F}(n)$, $E \subseteq I_q$.  
\end{itemize}

Note that all of these forcing statements are $X$-computable and are closed under extensions.  
If $q \in \mathbb{Q}_{e'}^{X \oplus G}$ for our generic $G$, then for each $n$, there is 
an index $k_n$ such that $F_{k_n} \Vdash (q$ is a condition up to level $n$) where $(F_{k_n},I_{k_n},S_{k_n})$ appears in the sequence defining $G$.  
However, if $q \not \in \mathbb{Q}_e^{X \oplus G}$, then this statement is not necessarily forced because, in order to keep our statements 
$X$-computable, we have not included conditions to handle divergence computations.  

The requirement $\mathcal{T}_{e',\mathcal{K},q}^{X,A^*(c),B^*(c)}$ consists of all finite 
transitive subtournaments $F$ of $T_e^X$ such that either 
\begin{enumerate}
\item[(C1)] $F \Vdash q \not \in \mathbb{Q}_{e'}^{X \oplus G}$; or

\item[(C2)] there is an $n \leq |F|$ such that $F \Vdash (q$ is a condition up to level $n$) and  
for all $E \in S_q^{X \oplus F}(n)$ and all partitions $E = E_0 \cup E_1$, there is an $i \in \{ 0,1 \}$ and a transitive $F' \subseteq E_i$ 
such that $\exists a \in A^*(c) \, \exists b \in B^*(c) \, (R_{\mathcal{K}}^{X \oplus F}(F_q \cup F',a,b))$.  
\end{enumerate}
The function $a^X_{\mathcal{T}_{e',\mathcal{K},q}}(F)$ is defined by $a^X_{\mathcal{T}_{e',\mathcal{K},q}}(F) = a_{\mathcal{K}}^{X \oplus F}(F_q)$.  

\begin{lem}
\label{lem:dense}
Let $G = \cup F_k$ be a generic defined by a sequence of conditions $(F_k,I_k,S_k)$ and let $e'$ be an index such that $T_{e'}^{X \oplus G}$ is an 
infinite tournament.  Each requirement $\mathcal{K}^{X \oplus G,A^*(c),B^*(c)}$ is 
uniformly dense in $\mathbb{Q}_{e'}^{X \oplus G}$.  
\end{lem}

\begin{proof}
Fix a requirement $\mathcal{K}^{X \oplus G,A^*(c),B^*(c)}$ and a condition $q = (F_q,I_q,S_q) \in \mathbb{Q}_{e'}^{X \oplus G}$ such that 
$\mathcal{K}^{X \oplus G}$ is essential below $q$.  Being essential below $q$ implies that $a_{\mathcal{K}}^{X \oplus G}(F_q)$ converges, 
and so there is an $i$ such that $a_{\mathcal{K}}^{X \oplus F_i}(F_q) = a_{\mathcal{K}}^{X \oplus G}(F_q)$ converges.  By definition, 
$a_{\mathcal{T}_{e',\mathcal{K},q}}^X(F_i) = a_{\mathcal{K}}^{X \oplus F_i}(F_q)$, so it also converges.  By Lemma \ref{lem:settle}, 
there is a condition $(F_k,I_k,S_k)$ settling $\mathcal{T}_{e',\mathcal{K},q}^{X,A^*(c),B^*(c)}$.  For notational simplicity, let 
$a = a_{\mathcal{T}_{e',\mathcal{K},q}}^X(F_k) = a_{\mathcal{K}}^{X \oplus G}(F_q)$.

By Definition \ref{def:settle}, there are two ways in which $(F_k,I_k,S_k)$ could settle $\mathcal{T}_{e',\mathcal{K},q}^{X,A^*(c),B^*(c)}$.  We consider these 
options separately.  First, we show that if $(F_k,I_k,S_k)$ settles $\mathcal{T}_{e',\mathcal{K},q}^{X,A^*(c),B^*(c)}$ because 
$F_k \in \mathcal{T}_{e',\mathcal{K},q}^{X,A^*(c),B^*(c)}$, then $\mathcal{K}^{X \oplus G,A^*(c),B^*(c)}$ 
satisfies the required condition for uniform density in Definition \ref{def:EMud} with respect to the condition $q$.  Second, we show that $(F_k,I_k,S_k)$ 
cannot settle $\mathcal{T}_{e',\mathcal{K},q}^{X,A^*(c),B^*(c)}$ in the case when $F_k \not \in \mathcal{T}_{e',\mathcal{K},q}^{X,A^*(c),B^*(c)}$, 
completing the proof.  

First, suppose $F_k \in \mathcal{T}_{e',\mathcal{K},q}^{X,A^*(c),B^*(c)}$.  Because $q \in \mathbb{Q}_{e'}^{X \oplus G}$, $F_k$ must be in 
$\mathcal{T}_{e',\mathcal{K},q}^{X,A^*(c),B^*(c)}$ because of Condition (C2).  However, replacing the 
oracle in (C2) by $X \oplus G$ and comparing the result with the definition of uniform density for $\mathcal{K}^{X \oplus G,A^*(c),B^*(c)}$ with respect to the 
condition $q$ shows that we have obtained exactly what we need.  

Second, suppose $(F_k,I_k,S_k)$ settles $\mathcal{T}_{e',\mathcal{K},q}^{X,A^*(c),B^*(c)}$ but 
$F_k \not \in \mathcal{T}_{e',\mathcal{K},q}^{X,A^*(c),B^*(c)}$.  By the definition of settling, 
we fix $x$ such that whenever $E \in S_k^X(n)$ is on an infinite path in $S_k$ and $F' \subseteq E$ is transitive, 
$F_k \cup F' \not \in \mathcal{T}_{e',\mathcal{K},q}^{X,a,(x,\infty)}$.  We derive a contradiction by constructing such an $F'$ (called $H'$ below) 
for which $F_k \cup F' \in \mathcal{T}_{e',\mathcal{K},q}^{X,a,(x,\infty)}$.

Because $\mathcal{K}^{X \oplus G}$ is essential below $q$, 
there is a level $n$ (for this fixed $x$) and a finite set $B > x$ such that for all $E \in S_q^{X \oplus G}(n)$ and all partitions $E = E_0 \cup E_1$, there is 
an $i \in \{ 0,1 \}$ and a transitive $F' \subseteq E_i$ such that $F_q \cup F' \in \mathcal{K}^{X \oplus G,a,B}$, and hence 
$F_q \cup F' \in \mathcal{K}^{X \oplus G,a,(x,\infty)}$.  That is, for each such $F'$
\[
\exists b > x \, ( R_{\mathcal{K}}^{X \oplus G}(F_q \cup F',a,b)).
\]
Let $H$ be a finite initial segment of $G$ which is long enough to force these statements for each of the finitely many $F'$ sets.  We also 
assume that $H$ is long enough that $S_q^{X \oplus H}(m)$ converges for all $m \leq n$ and $F_q$ is a transitive subset of 
$T_{e'}^{X \oplus H}$.

Let $H'  = H \setminus F_k$.  Because $G$ is a subset of some set coded by $S_k$, there is a level $m$ and an $E \subseteq S_k(m)$ which is on 
an infinite path through $S_k$ (namely the one containing $G$) such that $H' \subseteq E$.  Furthermore, because $G$ is transitive in 
$T_e^X$, so is $H'$.  Finally, since $F_k \cup H' = H$, we have for each $F'$
\[
\exists b > x (R_{\mathcal{K}}^{X \oplus (F_k \cup H')}(F_q \cup F',a,b))
\]
which means $F_k \cup H' \in \mathcal{T}_{e',\mathcal{K},q}^{X,a,(x,\infty)}$ giving the desired contradiction.
\end{proof}

We have now completed the description of the iteration forcing except for the proof of Theorem \ref{thm:settle}.  Note that we have met the 
required conditions: $G$ is an infinite transitive subtournament of $T_e^X$ such that $X \oplus G$ does not compute a solution to $c$ (by 
Lemma \ref{lem:nosolve}) and such that for each index $e'$ for which $T_{e'}^{X \oplus G}$ is an infinite tournament, each requirement 
$\mathcal{K}^{X \oplus G,A^*(c),B^*(c)}$ is uniformly dense (by Lemma \ref{lem:dense}).  

We end this subsection by proving Theorem \ref{thm:settle}.  Fix a requirement $\mathcal{K}^{X,A^*(c),B^*(c)}$ which is 
uniformly dense and a condition $(F,I,S)$.  We need to find a condition $(F',I',S') \leq (F,I,S)$ which settles $\mathcal{K}^{X,A^*(c),B^*(c)}$.  
Finding this condition breaks into two cases: when $\mathcal{K}^{X}$ is essential below $(F,I,S)$ and when 
$\mathcal{K}^{X}$ is not essential below $(F,I,S)$.  In each case, we will need a partition lemma for $S$.  

We begin with the case when $\mathcal{K}^{X}$ is essential below $(F,I,S)$.  We state the required partition lemma, 
show that it suffices and then return to the proof of the partition lemma.  For finite subtournaments $E$ and $E'$, we write 
$E \rightarrow E'$ if  $T_e^X(x,y)$ holds for all $x \in E$ and $y \in E'$.  

\begin{lem}
\label{family_split}
  Let $S$ be an infinite family of subtournaments.  Fix a level $n$ and an $E \in S(n)$ on an infinite path through $S$.  
  There is a partition $E=E_0\cup E_1$ and an infinite family $S'$ such that $E_0+S' \leq S$, $E_1+S' \leq S$, 
  and for all $m$ and all $E' \in S'(m)$, $E_0 \rightarrow E'$ and $E' \rightarrow E_1$.
\end{lem}

\begin{lem}
Let $\mathcal{K}^X$ be essential below $(F,I,S)$ and let $\mathcal{K}^{X,A^*(c),B^*(c)}$ be uniformly dense.  There is an $(\tilde{F},\tilde{I},\tilde{S})\leq (F,I,S)$ settling 
$\mathcal{K}^{X,A^*(c),B^*(c)}$.
\end{lem}

\begin{proof}
If $F \in \mathcal{K}^{X,A^*(c),B^*(c)}$, then there is nothing to show, so assume $F \not \in \mathcal{K}^{X,A^*(c),B^*(c)}$.  
By uniform density, we can fix a level $n$ such that whenever $E \in S(n)$ and 
$E = E_0\cup E_1$ is a partition, there is an $i\in\{0,1\}$ and a transitive $F' \subseteq E_i$ such that $F\cup F'\in \mathcal{K}^{X,A^*(c),B^*(c)}$.  
Fix $E \in S(n)$ such that $E$ is on an infinite path through $S$. Let $E = E_0 \cup E_1$ and $S'$ be the partition and infinite family guaranteed by Lemma 
\ref{family_split}.  Fix an $i \in \{ 0,1 \}$ and a transitive $F' \subseteq E_i$ such that $F\cup F'\in \mathcal{K}^{X,A^*(c),B^*(c)}$.   Since 
$F \not \in \mathcal{K}^{X,A^*(c),B^*(c)}$, we have $F' \neq \emptyset$.

Let $\tilde{F} = F \cup F'$ and $\tilde{S} = S'$.  To define $\tilde{I}$, let $c$ and $d$ be such that $I$ is the $F$-minimal interval $(c,d)$.  
Since $F'$ is transitive and hence is a finite linear order contained in $I$, let $a$ denote the minimal element of $F'$ and $b$ denote the maximal element 
of $F'$.  (If $|F'|=1$, then $a=b$.)  If $i=0$, then for all $m$ and all 
$E' \in S'(m)$, we have $F' \rightarrow E'$ and hence $E'$ is contained in the minimal interval $(b,d)$ of $\tilde{F}$.  In this case, set 
$\tilde{I} = (b,d)$.  If $i = 1$, then for all $m$ and all $E' \in S'(m)$, $E' \rightarrow F'$ and hence $E'$ is contained in the minimal interval 
$(c,a)$ of $\tilde{F}$.  In this case, set $\tilde{I} = (c,a)$.

It is clear that $(\tilde{F},\tilde{I},\tilde{S})$ is a condition and it settles $\mathcal{K}^{X,A,^*(c),B^*(c)}$ because $\tilde{F} = 
F\cup F'\in \mathcal{K}^{X,A^*(c),B^*(c)}$.  To see that $(\tilde{F},\tilde{I},\tilde{S})\leq (F,I,S)$, note that 
$\tilde{F} \setminus F = F' \subseteq E \in S(n)$, and 
$(\tilde{F} \setminus F) + S' = F'+S' \leq S$.
\end{proof} 

We next turn to the proof of Lemma \ref{family_split}.  For a family of subtournaments $S$, let $\bigcup S$ be the set of all 
$x$ such that $x \in E$ for some $E \in S(n)$.  

\begin{defn}
\label{defn:partition}
Let $S$ be a family of subtournaments.  A \emph{pointwise partition of} $S$ is a function 
$g: \bigcup S \rightarrow \{ 0,1 \}$.  
Given a partition $g$ of $S$, we say $g$ \emph{generates the functions} $S_0$ and $S_1$ defined by 
\begin{gather*}
S_0(n) = \{ E \mid \exists E' \in S(n) \, (x \in E \leftrightarrow (x \in E' \wedge g(x) = 0)) \} \\
S_1(n) = \{ E \mid \exists E' \in S(n) \, (x \in E \leftrightarrow (x \in E' \wedge g(x) = 1)) \}
\end{gather*}
where in both sets, $E$ ranges over the finite subtournaments of $T_e^X$.
\end{defn}

The functions $S_0$ and $S_1$ need not be families of potential subtournaments.  Condition (A2) in Definition \ref{def:family} 
can fail because a subtournament $E$ appears in both $S_0(n)$ and $S_0(n+1)$.  In particular, if $S$ is infinite, then $S_0$ and $S_1$ are 
also infinite in the sense that for $i \in \{ 0,1 \}$ and for all levels $n$, $S_i(n) \neq \emptyset$.  This fact follows because each  
$E' \in S(n)$ must be partitioned into an $S_0$-half and an $S_1$-half.  Note that one of these halves could be $\emptyset$, i.e.~we 
can have $\emptyset \in S_i(n)$.  

$S_0$ and $S_1$ do satisfy a condition similar to (A2):  
For every $E' \in S_i(n+1)$, either $E' \in S_i(n)$ or there is an $E'' \in S_i(n)$ such that $E'' \subsetneq E'$ 
and for all $x \in E' \setminus E''$, $x > \text{max}(S(n))$.  Therefore, $S_0$ and $S_1$ are finitely branching trees and 
we can treat $S_0$ and $S_1$ like infinite families with the 
exception that an infinite path in $S_i$ may code a finite set.  In particular, we write $S' \leq S_i$, for a family $S'$, to denote that 
for every level $n$ in $S'$, there is a level $m \geq n$ in $S_i$ such that for all $E' \in S'(n)$, there is an $E \in S_i(m)$ such that 
$E' \subseteq E$.  Note that if $S' \leq S_i$, then $S' \leq S$.  
 
In the next lemma, we show that we can always refine at least one of $S_0$ or $S_1$ to obtain a family of subtournaments contained in $S$.

\begin{lem}
\label{family_partition}
  Let $S$ be an infinite family of subtournaments and let $S_0$ and $S_1$ be generated by a pointwise partition $g$ of $S$. Then there is an infinite 
  family $S'\leq S_i$ for some $i\in\{0,1\}$.
\end{lem}

\begin{proof}
The proof splits into two cases.  For the first case, suppose that there is an $n$ and an $E_0 \in S_0(n)$ such that $E_0 \in S_0(m)$ for all $m \geq n$.  Then, 
for every $m \geq n$, there is an $E_1^m \in S_1(m)$ such that $E_0 \cup E_1^m \in S(m)$.  In particular, we have $g(x) = 0$ for all $x \in E_0$ and 
$g(x) = 1$ for all $x \in E_1^m$.  Define $S'(k)$ by 
\[
S'(k) = \{ E_1 \in S_1(n + k) \mid E_0 \cup E_1 \in S(n+k) \}.
\]
We show that $S'$ is an infinite family of potential subtournaments from which it immediately follows that $S' \leq S_1$.

Condition (A1) in Definition \ref{def:family} for $S'$ is clear, so consider (A2).  Fix $E_1 \in S'(\ell+1)$.  We need to find $E_1' \in S'(\ell)$ such that 
$E_1' \subsetneq E_1$ and for all $x \in E_1 \setminus E_1'$, $x > \text{max}(S'(\ell))$.  

Since $E_1 \in S'(\ell+1)$, we have $E_0 \cup E_1 \in S(n+\ell+1)$ and we can fix $E \in S(n +\ell)$ such that $E \subsetneq E_0 \cup E_1$ and for all 
$x \in (E_0 \cup E_1) \setminus E$, $x > \text{max}(S(n+\ell))$.  Since $E_0 \in S_0(n+\ell)$, and hence is contained in some element of $S(n+\ell)$, every 
element $x \in E_0$ satisfies $x \leq \text{max}(S(n+\ell))$.  Therefore, $E_0 \subseteq E$.  Set $E_1' = E \setminus E_0$ and we check that $E_1'$ has the 
desired properties.

To see that $E_1' \subseteq E_1$, notice that $E_1' \subseteq E \subseteq E_0 \cup E_1$ and $E_1' \cap E_0 = \emptyset$.  Therefore, $E_1' \subseteq E_1$.  
Furthermore, since $E_0 \subseteq E \subsetneq E_0 \cup E_1$, it follows that $E_1' \subsetneq E_1$.  If $x \in E_1 \setminus E_1'$, then 
$x \in (E_0 \cup E_1) \setminus E$ and hence $x > \text{max}(S(n+\ell)) \geq \text{max}(S'(\ell))$.  
To see that $E_1' \in S'(\ell)$, note that $E_1' \subseteq E_1$ implies that $g(x) = 1$ for all $x \in E_1'$.  Since $E \in S(n+\ell)$ and $E = E_0 \cup E_1'$, it follows 
that $E_1' \in S_1(n+\ell)$ and $E_1' \in S'(\ell)$.   

To show that $S'$ is infinite, it suffices to show that for every $k$, there is an $E \in S(n+k)$ such that $E_0 \subseteq E$ and 
$E \setminus E_0 \in S_1(n+k)$.  If this property failed for some $k$, then for all $E_0' \in S_0(n+k)$ with $E_0 \subseteq E_0'$, we would have 
$E_0 \subsetneq E_0'$ and hence $E_0 \not \in S_0(n+k)$, contrary to our case assumption.    

Now, assume that the first case fails.  Then, for each level $n$ and each $E_0 \in S_0(n)$, there is a level $m > n$ such that 
$E_0 \not \in S_0(m)$.  (This can happen either because $E_0$ has been properly extended or because the corresponding branch in $S$ has been 
eliminated.)  Because the levels of $S_0$ are finite, it follows that for each level $n$, there is a level $m$ such that for all $E_0 \in S_0(n)$, 
$E_0 \not \in S_0(m)$.  (Recall that $S_0$ is infinite because $S$ is infinite.)  We 
define $S' \leq S_0$ inductively.  Let $S'(0) = S_0(0)$.  Given $S'(n) = S_0(n_0)$, let $S'(n+1) = S_0(m)$ for the least $m > n_0$ such that 
for every $E_0 \in S_0(n_0)$, $E_0 \not \in S_0(m)$.  The fact that $S'$ is an infinite family of subtournaments follows almost immediately 
from this definition.
\end{proof}

Definition \ref{defn:partition} and Lemma \ref{family_partition} can be extended to pointwise partitions of $S$ into any fixed finite number of pieces.  

\begin{defn}
Let $S$ be a family of subtournaments and let $E \in S(n)$.   $S\upharpoonright E$ is the family of 
subtournaments defined by 
\[
(S\upharpoonright E)(m) = \{ E' \mid E' \cap E = \emptyset \wedge E' \cup E \in S(n+m+1) \wedge \forall x \in E' (x > \text{max}(S(n))) \}.
\]
As a tree, $S\upharpoonright E$ is formed by taking the subtree of $S$ above $E$ and removing the set $E$ from each node.  
\end{defn}

If $E \in S(n)$ is on an infinite path through $S$, then $S \upharpoonright E$ is an infinite family of subtournaments.  Note that 
$E < E'$ for all $E' \in (S\upharpoonright E)(0)$ and that $E + (S\upharpoonright E)$ corresponds to the subtree of $S$ above $E$.  
We can now give the proof of Lemma \ref{family_split}.

\begin{proof}
Fix an infinite family $S$ and a node $E \in S(n)$ which is on an infinite path through $S$.  
Let $P_E$ be the (finite) set of (ordered) partitions of $E$ defined by 
\[
P_E = \{ \langle E_0, E_1 \rangle \mid E_0 \cup E_1 = E \text{ and } E_0 \cap E_1 = \emptyset \}.
\]
We define a pointwise partition $g$ of the infinite family $S \upharpoonright E$.  Let $g$ be the function from $\bigcup (S \upharpoonright E)$ into 
$P_E$ defined by $g(x) = \langle E_0, E_1 \rangle$ where 
\[
E_0 = \{ a \in E \mid T_e^X(a,x) \text{ holds} \} \text{ and } E_1 = \{ b \in E \mid T_e^X(x,b) \text{ holds} \}.
\]
By the extended version of Lemma \ref{family_partition}, there is an infinite family $S'$ such that $S' \leq S_{\langle E_0,E_1 \rangle}$ 
for a fixed partition $\langle E_0,E_1 \rangle$.  Since $E < E'$ for all $E' \in S'(0)$, we have that $E_0 + S' \leq S$ and $E_1+S' \leq S$ are infinite families. 

Fix $m$ and $E' \in S'(m)$.  Since $S' \leq S_{\langle E_0,E_1 \rangle}$, there is an $k \geq m$ and an 
$\tilde{E} \in S_{\langle E_0, E_1 \rangle}(k)$ such that $E' \subseteq \tilde{E}$.  It follows that for all $x \in E'$, $a \in E_0$ and $b \in E_1$, we have 
$T_e^X(a,x)$ and $T_e^X(x,b)$.  Therefore, $E_0 \rightarrow E'$ and $E' \rightarrow E_1$ as required.  
\end{proof}

Lastly, we turn to the remaining case in the proof of Theorem \ref{thm:settle}, when $\mathcal{K}^{X}$ is not essential below 
$(F,I,S)$ and $a_{\mathcal{K}}^X(F)$ converges.  

\begin{lem}
If $\mathcal{K}^{X}$ is not essential below $(F,I,S)$ and $a_{\mathcal{K}}^X(F)$ converges, 
then there is an infinite family $S'' \leq S$ such that $(F,I,S'') \leq (F,I,S)$ settles $\mathcal{K}^{X,A^*(c),B^*(c)}$.
\end{lem}

\begin{proof}
Assume $\mathcal{K}^{X}$ is not essential below $(F,I,S)$ and $a_{\mathcal{K}}^X(F)$ converges.  
Fix $x$ such that for all levels $n$ and all $B > x$, there is an $E \in S(n)$ and a partition $E_0 \cup E_1 = E$ such that for all $i \in \{ 0, 1 \}$ and all 
transitive $F' \subseteq E_i$, $F \cup F' \not \in \mathcal{K}^{X, a_{\mathcal{K}}^X(F),B}$.   Because the dependence on $B$ in 
$\mathcal{K}^{X, a_{\mathcal{K}}^X(F),B}$ is positive, we can restrict our attention to sets $B$ of the form $(x,x+v+2)$ for all $v$.  Before 
proceeding, we examine this hypothesis in more detail.  

Let $Q(u,v)$ be the predicate such that for any finite tournament $E$ and any $v$, $Q(E,v)$ holds if and only if 
\[
\exists \text{ partition } E_0 \cup E_1 = E \, \forall i \in \{ 0,1 \} \, \forall \text{ transitive } F' \subseteq E_i \, ( F \cup F' \not \in 
\mathcal{K}^{X,a_{\mathcal{K}}^X(F),(x,x+v+2)} ).
\]
Because  
\[
F \cup F' \not \in \mathcal{K}^{X,a_{\mathcal{K}}^X(F),(x,x+v+2)} \Leftrightarrow \forall b \, 
\big( x < b < x+v+2 \rightarrow \neg R_{\mathcal{K}}^X(F \cup F', a_{\mathcal{K}}^X(F),b) \big)
\]
the predicate $Q(E,v)$ is $X$-computable.  We write $Q(E,\infty)$ for the same predicate with the inequality $x < b < x+v+2$ replaced by $x < b$ and 
note that $Q(E,\infty)$ is a $\Pi^{0,X}_1$ predicate.  

We can restate our hypothesis in terms of $Q$.  For every level $n$ and for every $v$, 
there exists $E \in S(n)$ such that $Q(E,v)$.  Since each level $S(n)$ is finite, we have that for every level $n$ there exists $E \in S(n)$ such that 
$Q(E,\infty)$.

Suppose $E \in S(n+1)$ is a successor of $\tilde{E} \in S(n)$.  If $Q(E,v)$ holds (allowing the possibility that $v = \infty$), then $Q(\tilde{E},v)$ holds as 
well because the witnessing partition $E_0 \cup E_1 = E$ restricts to a witnessing partition $\tilde{E}_0 \cup \tilde{E}_1 = \tilde{E}$ for 
$Q(\tilde{E},v)$.  Therefore, we can define two subtrees, $S'$ and $\tilde{S}$, of $S$ as follows:
\begin{eqnarray*}
S'(n) & = & \{ E \in S(n) \mid Q(E,n) \text{ holds} \} \\
\tilde{S}(n) & = & \{ E \in S(n) \mid Q(E,\infty) \text{ holds} \}.
\end{eqnarray*}
Because the relation $Q(E,n)$ is $X$-computable, $S' \leq S$ is an infinite family of subtournaments.  $\tilde{S}$ satisfies all the requirements for 
being an infinite subtournament except it is not $X$-computable.  However, if $E \in S'(n)$ is on an infinite path through $S'$ then 
$E \in \tilde{S}(n)$.  

To define $S''$, we will partition $S'$.  
We consider each $E \in S'(n)$ and look at all partitions $E_0 \cup E_1 = E$ such that for all $i \in \{ 0 ,1 \}$ and all transitive 
$F' \subseteq E_i$, $F \cup F' \not \in \mathcal{K}^{X,a_{\mathcal{K}}^X(F),(x,x+n+2)}$.  We will form a tree $T_0$ of all ``left halves'' of such splittings 
(and hence implicitly also a tree of all ``right halves'' of such splittings) where we choose the ``left halves'' in a coherent manner.  Then, we 
show how to define an appropriate infinite family $S'' \leq S$ from $T_0$ which settles $\mathcal{K}^{X,A^*(c),B^*(c)}$.  

Formally, we proceed as follows.  Let $R(u,v,w,z)$ be the $X$-computable predicate such that $R(E_0,E_1,E,n)$ holds if and only if 
$E \in S'(n)$ and $E_0 \cup E_1 = E$ is a partition such that for all $i \in \{ 0,1 \}$ and all transitive $F' \subseteq E_i$, 
$F \cup F' \not \in \mathcal{K}^{X,a_{\mathcal{K}},(x,x+n+2)}$.  Notice that $R$ is symmetric in the $E_0$ and $E_1$ variables and that for all 
$E \in S(n)$, $Q(E,n)$ holds if and only if there are $E_0$ and $E_1$ such that $R(E_0,E_1,E,n)$ holds.
We define the tree $T_0$ inductively, starting with placing a root in $T_0$.  

To define the nodes at level $1$ in $T_0$, consider each $E \in S'(0)$ in turn.  Find the set of all (unordered) partitions $\{ E_0, E_1 \}$ 
such that $R(E_0,E_1,E,0)$ holds (and hence also $R(E_1,E_0,E,0)$ holds).  For each such $\{ E_0, E_1 \}$, add a node $\sigma$ at level $1$ to $T_0$ 
and label this node by an arbitrarily chosen element of $\{ E_0, E_1 \}$, suppose it is $E_0$.  We indicate this labeling by 
$E_{\sigma} = E_0$.  Also, include a second label for $\sigma$ indicating the set $E \in S'(0)$ which has been split.  We 
write $S_{\sigma} = E$ to indicate this information.   

To define the nodes at level $n+1$ (for $n > 0$), consider each $E \in S'(n)$ in turn.  Fix $E \in S'(n)$ and let $E' \in S'(n-1)$ be the 
predecessor of $E$ in the tree $S'$.  Find the set of all (unordered) partitions  
$\{ E_0, E_1 \}$ such that $R(E_0,E_1,E,n)$ holds.  Consider each of these sets in turn.  

The partition $E_0 \cup E_1 = E$ restricts to a partition $E'_0 \cup E'_1 = E'$ such that $R(E'_0,E'_1,E',n-1)$ holds.  By induction, there is a node 
$\delta$ at level $n$ of $T_0$ such that $S_{\delta} = E'$ and either $E_{\delta} = E_0'$ or $E_{\delta} = E_1'$.  
Add a node $\sigma$ to $T_0$ as a successor of $\delta$.  Set $S_{\sigma} = E$ and $E_{\sigma} = E_i$ 
where $i$ is chosen such that $E_{\delta} = E_i'$.    

This completes the description of $T_0$.  Notice that $T_0$ is an infinite finitely branching tree. 
For any nodes $\delta$ and $\sigma$ such that $\sigma$ is a 
successor of $\delta$, we have $E_{\delta} \subseteq E_{\sigma}$ and $S_{\delta} \subseteq S_{\sigma}$ is 
the predecessor of $S_{\sigma}$ in $S'$.  It follows easily by induction that if $\tau$ is an extension of $\sigma$ on $T_0$ with 
$E_{\sigma} = E_0$, $E_0 \cup E_1 = S_{\sigma}$, $E_{\tau} = E_0'$ and $E_0' \cup E_1' = S_{\tau}$, then 
$E_0 \subseteq E_0'$ and $E_1 \subseteq E_1'$.  

Furthermore, we claim that if $E_0 = E_{\sigma}$, $E_1 \cup E_0 = S_{\sigma}$ and 
$\sigma$ is on an infinite path through $T_0$, then for each $i \in \{ 0 ,1 \}$ and all transitive $F' \subseteq E_i$, 
$F \cup F' \not \in \mathcal{K}^{X,a_{\mathcal{K}}^X(F),(x,\infty)}$.  (Below, we refer to this claim as the \textit{main claim}.) 
To prove this claim, suppose for a contradiction that there is a 
transitive $\tilde{F} \subseteq E_i$ such that $F \cup \tilde{F} \in \mathcal{K}^{X,a_{\mathcal{K}}^X(F),(x,\infty)}$.  Fix $k$ such that 
$F \cup \tilde{F} \in \mathcal{K}^{X,a_{\mathcal{K}}^X(F),(x,x+k+2)}$ and let $\tau \in T_0$ be an extension of $\sigma$ at level $k$.  
Let $E_0' = E_{\tau}$ and let $E_1' \cup E_0' = S_{\tau}$.  Because $\tau$ is on $T_0$, for every transitive $F' \subseteq E_i'$, 
$F \cup F' \not \in \mathcal{K}^{X,a_{\mathcal{K}}^X(F),(x,x+k+2)}$.  However, because $\tau$ extends $\sigma$, we have 
$E_0 \subseteq E_0'$ and $E_1 \subseteq E_1'$.  Therefore, $\tilde{F} \subseteq E_i \subseteq E_i'$ is transitive and 
$F \cup \tilde{F} \in \mathcal{K}^{X,a_{\mathcal{K}}^X(F),(x,x+k+2)}$, giving the desired contradiction.

It remains to extract the infinite family $S'' \leq S$ which settles $\mathcal{K}^{X,A^*(c),B^*(c)}$.  
This extraction breaks into two cases.  

For the first case, assume that there is a level $n$ and an $E_0$ such that for every level $m \geq n$, there is a node $\sigma \in T_0$ 
at level $m$ such that $E_{\sigma} = E_0$.  Fix $n$ and $E_0$.  Because $T_0$ is finitely branching, there must be a node $\delta \in T_0$ at level $n$ 
such that for all $m \geq n$, there is an extension $\sigma$ of $\delta$ at level $m$ such that $E_{\sigma} = E_0$.  Fix such a $\delta$.  
Define $S''$ to be the family such that $S''(k)$ is the set of all $E_1$ such that there is a node $\sigma \in T_0$ extending $\delta$ at level 
$n+k$ such that $E_{\sigma} = E_0$ and $S_{\sigma} = E_0 \cup E_1$.  That is, $S''(k)$ is the set of all ``right halves" of the splits of 
elements of $S'(n+k)$ for which $E_0$ was the ``left half".  

We claim that $S''$ is the desired family.  We check the required properties.  (Note that Property (A1) in Definition \ref{def:family} is immediate.)
\begin{itemize}

\item $S'' \leq S$:  Fix $E_1 \in S''(k)$.  Then, $E_0 \cup E_1 \in S'(n+k)$ and hence $E_0 \cup E_1 \in S(n+k)$.  Therefore, 
$E_1 \subseteq E$ for some $E \in S(n+k)$.

\item $S''$ is infinite:  Fix $k$ and we show that $S''(k)$ is nonempty.  By assumption, there is a node $\sigma \in T_0$ extending $\delta$ at 
level $n+k$ with $E_{\sigma} = E_0$.  Therefore, $S_{\sigma} \setminus E_0 \in S''(k)$.

\item  Property (A2) holds:  Let $E_1 \in S''(k+1)$ and fix $\sigma \in T_0$ extending $\delta$ at level $n+k+1$ be such that 
$E_0 \cup E_1 = S_{\sigma}$.  Let $\tau$ be the predecessor of $\sigma$ on $T_0$.  Since $\delta \subseteq \tau \subseteq \sigma$ and 
$E_{\delta} = E_{\sigma} = E_0$, we have $E_{\tau} = E_0$.   Therefore, $E_1' = S_{\tau} \setminus E_0 \in S''(k)$.  By construction, 
$S_{\tau} = E_0 \cup E_1'$ is the predecessor of $S_{\sigma} = E_0 \cup E_1$ in $S'$ and in $S$.  Therefore, $E_1' \subseteq E_1$ and the elements of 
$E_1 \setminus E_1' = (E_0 \cup E_1) \setminus (E_1' \cup E_0)$ are greater than $\text{max}(S(n+k))$ and hence greater than $\text{max}(S''(k))$.  

\item $(F,I,S'')$ settles $\mathcal{K}^{X,A^*(c),B^*(c)}$:  Fix $E_1 \in S''(k)$ such that $E_1$ is on an infinite path through $S''$.  Let 
$\sigma \in T_0$ be the node witnessing that $E_1 \in S''(k)$.  Then $\sigma$ is on an infinite path through $T_0$ and 
$E_0 \cup E_1 = S_{\sigma}$.  By the main claim, for all transitive 
$F' \subseteq E_1$, $F \cup F' \not \in \mathcal{K}^{X, a_{\mathcal{K}}^X(F),(x,\infty)}$.
\end{itemize}

For the second case, assume that for any level $n$ and any $E_0 = E_{\sigma}$ for some $\sigma \in T_0$ at level $n$, there is a level $m > n$ such that 
$E_{\tau} \neq E_0$ for all $\tau \in T_0$ at level $m$.  Define $S''(k)$ inductively as follows.  $S''(0)$ is the set of all $E_0$ such that 
$E_0 = E_{\sigma}$ for some $\sigma \in T_0$ at level $0$.  Assume that $S''(k)$ has been defined as the set of all $E_0$ such that $E_0 = E_{\sigma}$ 
for some $\sigma \in T_0$ at level $\ell_k$.  Let $\ell_{k+1}$ be the first level in $T_0$ such that for every $E_0 \in S''(k)$ and every $\tau \in T_0$ at 
level $\ell_{k+1}$, $E_{\tau} \neq E_0$.  By our case assumption and the fact that $T_0$ is finitely branching, $\ell_{k+1}$ is defined.  Let 
$S''(k+1)$ be the set of all $E_0$ such that $E_0 = E_{\tau}$ for some $\tau \in T_0$ at level $\ell_{k+1}$.  

We claim that $S''$ is the desired family.  We check the required properties.  (Again, (A1) holds trivially.)
\begin{itemize}

\item $S'' \leq S$:  If $E_0 \in S''(k)$, then $E_0 = E_{\sigma}$ for some $\sigma \in T_0$ at level $\ell_k$.  Let $E_1 = S_{\sigma} \setminus E_0$.  Then, 
$E_0 \cup E_1 \in S'(\ell_k)$ and $E_0 \cup E_1 \in S(\ell_k)$.  Therefore, $E_0 \subseteq E$ for some $E \in S(\ell_k)$.

\item $S''$ is infinite:  By our case assumption, $\ell_k$ is defined for all $k$.  Since $T_0$ is infinite, $S''(k) \neq \emptyset$ for all $k$.

\item Property (A2): Let $E_0 \in S''(k+1)$ and fix $\sigma \in T_0$ at level $\ell_{k+1}$ such that $E_{\sigma} = E_0$.  Let $\tau \in T_0$ at 
level $\ell_k$ be such that $\tau \subseteq \sigma$.  Then $E_{\tau} \in S''(k)$ and $E_{\tau} \subsetneq E_{\sigma}$.  The elements in 
$E_{\sigma} \setminus E_{\tau}$ are all greater than the elements in $S(\ell_k)$ and hence are all greater than the elements in $S''(k)$.  

\item $(F,I,S'')$ settles $\mathcal{K}^{X,A^*(c),B^*(c)}$:  Fix $E_0 \in S''(k)$ such that $E_0$ is on an infinite path through $S''$.  Let $\sigma \in T_0$ 
at level $\ell_k$ be such that $E_0 = E_{\sigma}$.  Because $\sigma$ is on an infinite path through $T_0$, it follows by the main claim 
that all transitive $F' \subseteq E_0$ satisfy $F \cup F' \not \in \mathcal{K}^{X,a_{\mathcal{K}}^X(F),(x,\infty)}$.  
\end{itemize}
\end{proof}

\subsection{Ground forcing}

We now carry out the ground level forcing to produce the coloring $c$.  Our forcing conditions are triples $(c,A^*,B^*)$ where $c$ is coloring of two 
element subsets of a finite domain $[0,|c|]$, $A^*$ and $B^*$ are subsets of $[0,|c|]$ and $A^*\cap B^*=\emptyset$.  
We say $(c,A^*,B^*) \leq (c_0,A^*_0,B^*_0)$ if:
\begin{itemize}
  \item $c_0\subseteq c$,
  \item $A^*_0\subseteq A^*$,
  \item $B^*_0\subseteq B^*$,
  \item Whenever $a\in A^*_0$ and $x>|c_0|$, $c(a,x)=0$,
  \item Whenever $b\in B^*_0$ and $x>|c_0|$, $c(b,x)=1$.
\end{itemize}

Clearly the set of $(c,A^*,B^*)$ such that $i\in A^*\cup B^*$ is dense, so we may ensure that the coloring given by a generic is stable.  We 
need to ensure that our generic coloring does not compute a solution to itself.  We say 
$(c,A^*,B^*) \Vdash (\Phi_e^G$ is finite) if 
\[
\exists k \, \forall (c_0,A_0^*,B_0^*) \leq (c,A^*,B^*) \, \forall x \, (\Phi_e^{c_0}(x) =  1 \rightarrow x \leq k).
\]
We say $(c,A^*,B^*) \Vdash ( \Phi_e^G \not \subseteq A^*(G) \wedge \Phi_e^G \not \subseteq B^*(G))$ if 
\[
\exists a \in A^* \, \exists b \in B^* \, (\Phi_e^c(a) = \Phi_e^c(b) = 1).
\]

\begin{lem}
For each index $e$, the set of conditions which either force $\Phi_e^G$ is finite or force $\Phi_e^G \not \subseteq A^*(G) \wedge \Phi_e^G \not \subseteq 
B^*(G)$ is dense.  
\end{lem}

\begin{proof}
Fix an index $e$ and a condition $(c,A^*,B^*)$.  If some extension of $(c,A^*,B^*)$ forces $\Phi_e^G$ is finite, then we are done.  Otherwise, since 
$(c,A^*,B^*)$ does not force $\Phi_e^G$ is finite, there is an $x > |c|$ and a condition $(c_0,A^*,B^*)$ extending $(c,A^*,B^*)$ such that 
$\Phi_e^{c_0}(x) = 1$.  (Without loss of generality, only the coloring changes.)  Since $(c_0, A^* \cup \{ x \}, B^*) 
\leq (c,A^*,B^*)$,  $(c_0, A^* \cup \{ x \}, B^*)$ does not force $\Phi_e^G$ is finite.  So there is a $y > |c_0|$ and a condition $(c_1, A^* \cup \{ x \}, B^*)$ 
extending $(c_0, A^* \cup \{ x \}, B^*)$ such that $\Phi_e^{c_1}(y) = 1$.   The condition $(c_1, A^* \cup \{ x \}, B^* \cup \{ y \})$ extends 
$(c,A^*,B^*)$ and forces 
$\Phi_e^G \not \subseteq A^*(G) \wedge \Phi_e^G \not \subseteq B^*(G)$.
\end{proof}

Finally, we need to force the requirements $\mathcal{K}^{G,A^*(G),B^*(G)}$ for any generic $G$ to be uniformly dense in $\mathbb{Q}_e^G$.  Fix an 
index $e$ and a potential iterated forcing condition $p = (F_p,I_p,S_p)$, where $F_p$ is a finite set, $I_p$ is a pair of elements in $F_p$ and $S_p$ is 
the index for a potential family of subtournaments of $T_e^G$.  We define the following forcing notions.
\begin{itemize}
\item $(c,A^*,B^*) \Vdash (F_p$ is a transitive subtournament of $T_e^G)$ if 
for every $u,v \in F_p$, $\Phi_e^c(u,v)$ converges and the induced structure on $F_p$ makes it a transitive tournament.  

\item $(c,A^*,B^*) \Vdash (F_p$ is a not transitive subtournament of $T_e^G)$ if there is no extension $(c_0,A_0^*, B_0^*) \leq (c,A^*,B^*)$ forcing 
$F_p$ is a transitive subtournament of $T_e^G$.  

\item $(c,A^*,B^*) \Vdash (S_q^G$ is not total) if there is an $\ell$ such that for every $(c_0,A_0^*,B_0^*) \leq (c,A^*,B^*)$, $S_q^{c_0}(\ell)$ diverges.

\item $(c,A^*,B^*) \Vdash (p \not \in \mathbb{Q}_e^G)$ if any of the following conditions hold
\begin{itemize}
\item $(c,A^*,B^*) \Vdash (F_p$ is not a transitive subtournament of $T_e^G)$; or  
\item $(c,A^*,B^*) \Vdash (F_p$ is a transitive subtournament of $T_e^G)$ but $I_p$ is not a minimal interval in $F_p$; or
\item $(c,A^*,B^*) \Vdash (S_q^G$ is not total); or 
\item $S_q^c(0) = \{ \tilde{F}_0, \ldots, \tilde{F}_k \}$ and $F \not < \tilde{F}_j$ for some $j$; or 
\item there is an $n$ such that $S_q^c(n)$ and $S_q^c(n+1)$ both converge, but they violate the (A2) condition.
\end{itemize}

\item $(c,A^*,B^*) \Vdash (\mathcal{K}^{G}$ is not essential below $p)$ if for every completion $(\tilde{c},A^*(\tilde{c}),B^*(\tilde{c}))$ of 
$(c,A^*,B^*)$ to a stable $2$-coloring of $\omega$ for which $\mathcal{K}^{\tilde{c},A^*(\tilde{c}),B^*(\tilde{c})}$ is a requirement and 
$p \in \mathbb{Q}_e^{\tilde{c}}$, $\mathcal{K}^{\tilde{c}}$ is not essential below $p$. 

\end{itemize}

\begin{lem}
Let $\mathcal{K}^{G,A^*(G),B^*(G)}$ be a potential requirement given by the indices $i$ and $i'$.  Then for any potential iterated forcing condition 
$p$, there is a dense set of conditions $(c,A^*,B^*)$ such that:
\begin{itemize}
\item $(c,A^*,B^*)\Vdash p\not\in\mathbb{Q}^G_e$; or
\item $(c,A^*,B^*)\Vdash\mathcal{K}^{G}$ is not essential below $p$; or
\item there is a level $n$ such that $S_p^c(n)$ converges and whenever $E \in S_p^c(n)$ and $E = E_0\cup E_1$ is a partition, there is an $i\in\{0,1\}$ and a 
transitive $F' \subseteq E_i$ such that 
\[
\exists a \in A^* \, \exists b \in B^* \, ( \Phi_i^c(F_p \cup F',a,b) = 1 ).
\]
\end{itemize}
\end{lem}

\begin{proof}
Fix a condition $(c,A^*,B^*)$ and a potential iterated forcing condition $p=(F_p,I_p,S_p)$.  If there is any $(c',A',B')\leq(c,A^*,B^*)$ forcing that 
$p\not\in\mathbb{Q}^c_e$ then we are done, so assume not.

Suppose there is an extension $(c',A^*,B^*)\leq (c,A^*,B^*)$ such that $a^{c'}_{\mathcal{K}}(F_p)$ converges, 
a finite set $B> \max(A^* \cup B^* \cup \{ a^{c'}_{\mathcal{K}}(F_p) \})$ and an $n$ such that 
$S_p^{c'}(n)$ converges and whenever $E \in S_p^{c'}(n)$ and 
$E = E_0\cup E_1$ is a partition, there is an $i\in\{0,1\}$ and a transitive $F' \subseteq E_i$ such that 
\[
\exists b \in B \, (\Phi_i^{c'}(F_p \cup F',a^{c'}_{\mathcal{K}}(F_p),b) = 1)
\]
i.e.~$F_p \cup F' \in\mathcal{K}^{c',a^c_{\mathcal{K}}(F),B}$.  Without loss of generality, we may assume $a^{c'}_{\mathcal{K}}(F_p) \in A^*$ or 
$a^{c'}_{\mathcal{K}}(F_p) \in B^*$.  For simplicity, we assume 
$a^{c'}_{\mathcal{K}}(F_p) \in A^*$ as the other case is symmetric.  Then $(c',A^*,B^*\cup B)$ is the desired condition.

Suppose there is no such $(c',A^*,B^*)$.  Then we claim that $(c,A^*,B^*)$ already forces that $\mathcal{K}^{G}$ is not essential below $p$.   
Let $\tilde c$ be any completion of $c$ to a stable $2$-coloring on $\omega$, and suppose $\mathcal{K}^{c}$ were essential below $p$.  
Then there would be some $B>\max(A^*\cup B^*\cup\{a^c_{\mathcal{K}}(F_p)\})$ and an $n$ such that $S^c_p(n)$ converges and whenever $E\in S^c_p(n)$, every partition is as described above.  In particular, there would be some finite initial segment of $\tilde c$ witnessing the necessary computations, contradicting our assumption.
\end{proof}

The ideal to prove Theorem \ref{thm:EM} is constructed in the same way as the ideal to prove Theorem \ref{thm:ADS}.

\section{Conclusion}

We end with a few questions for future investigation.  \cite{Hir07} introduced two other principles we have not discussed:

\medskip
($\system{CRT}^2_2$) Cohesive Ramsey’s Theorem for pairs: Every $2$-coloring $c$ of $[\mathbb{N}]^2$ has an infinite set $S$ such that $\forall x\in S\exists y\forall z\in S[z>y\rightarrow c(x,y)=c(x,z)]$

\medskip
($\system{CADS}$) Cohesive $\ADS$: Every linear order has either an infinite subset which has order-type $\omega$, $\omega^*$, or $\omega+\omega^*$.

\medskip
They show that $\COH$ implies $\system{CRT}^2_2$ which in turn implies $\system{CADS}$.  Both reversals remain open:
\begin{question}
  Does $\COH$ imply $\system{CRT}^2_2$?  Does $\system{CADS}$ imply $\system{CRT}^2_2$?
\end{question}

While \cite{Cho:TA} recently showed that $\SRT$ does not imply $\RT$ using a nonstandard model, the following remains open:
\begin{question}
  Does $\SRT$ imply $\RT$ in $\omega$-models?
\end{question}
Their method requires constructing the entire Turing ideal satisfying $\SRT$ simultaneously in a way similar to the method used in this paper.  There is an essential difference, however: in the construction in \cite{Cho:TA}, the instance of $\RT$ and all solutions are low.  Since the low sets are all known in advance, they can construct an instance of $\RT$ by diagonalizing against all low sets, and then separately construct a Turing ideal satisfying $\SRT$ and consisting only of low sets.

In this paper, on the other hand, we constructed a Turing ideal by staying inside a collection of sets which could only be identified after a particular instance of $\SCAC$ or $\SRT$ had been chosen.  This required that the diagonalization step anticipate what may happen in the construction of the Turing ideal (by forcing certain sets to be generic in a future forcing notion).  This is unusual---most similar arguments in reverse mathematics carry out a diagonalization step against a collection of sets which have an a priori definition.  If this is necessary, it indicates a difference in kind between the non-implications in this paper and others in the literature.  To make this precise, we need to decide what we mean by an a priori collection of sets.  One way is the following:
\begin{question}
  Is there a set of Turing degrees $\mathcal{S}$ definable without parameters (in first-order logic) from $\leq$ and jump, together with an instance $(M,\preceq_M)$ of $\SCAC$ belonging to $\mathcal{S}$ such that:
  \begin{itemize}
  \item No element of $\mathcal{S}$ is a solution to $(M,\preceq_M)$, and
  \item There is a Turing ideal $\mathcal{I}\subseteq\mathcal{S}$ containing $(M,\preceq_M)$ and satisfying $\ADS$?
  \end{itemize}
\end{question}
All other known non-implications between combinatorial principles weaker than or similar to $\RT$ have positive solutions to the corresponding variant of this question.

\end{document}